\documentclass[11pt]{article}
\usepackage{setspace}
\usepackage{amsmath,amssymb}
\usepackage{amsthm}
\usepackage{algorithm}
\usepackage[noend]{algorithmic}
\usepackage{url}
\usepackage{fullpage}
\usepackage{makeidx}
\usepackage{enumerate}
\usepackage{graphicx,float,psfrag,epsfig}
\usepackage{color}

\numberwithin{equation}{section}

\newtheoremstyle{myexample} 
    {\topsep}                    
    {\topsep}                    
    {\rm }                   
    {}                           
    {\bf }                   
    {.}                          
    {.5em}                       
    {}  

\newtheoremstyle{myremark} 
    {\topsep}                    
    {\topsep}                    
    {\rm}                        
    {}                           
    {\bf}                        
    {.}                          
    {.5em}                       
    {}  

\newtheorem{claim}{Claim}[section]
\newtheorem{lemma}[claim]{Lemma}

\newtheorem{theorem}{Theorem}

\newtheorem{definition}[claim]{Definition}

\theoremstyle{myremark}
\newtheorem{remark}{Remark}[section]

\theoremstyle{myremark}

\theoremstyle{myexample}
\newtheorem{example}[remark]{Example}

\def\<{\langle}
\def\>{\rangle}

\def\bfe{{\boldsymbol{e}}}
\def\bone{{\boldsymbol{1}}}

\def\eps{{\varepsilon}}

\def\id{{\rm I}}
\def\sT{{\sf T}}

\def\OPT{{\sf OPT}}

\def\rank{{\rm rank}}

\def\P{{\mathbb P}}
\def\prob{{\mathbb P}}
\def\integers{{\mathbb Z}}
\def\naturals{{\mathbb N}}
\def\E{{\mathbb E}} 
\def\Var{{\sf{Var}}}

\def\one{\mathbf{1}}

\def\Var{{\sf Var}}

\def\cL{{\cal L}}

\def\Tr{{\sf {Tr}}}

\def\de{{\rm d}}

\def\cG{\mathcal{G}}
\def\d{{\mathrm{d}}}

\def\Var{{\sf Var}}

\def\spec{{\rm spec}}

\def\one{{\bf 1}}

\def\Var{{\rm Var}}

\def\rvec{{\rm vec}}
\def\maximize{{\rm maximize}}
\def\subjectto{{\rm subject to}}
\def\for{{\rm for}}

\def\reals{\mathbb{R}}

\def\bD{{\boldsymbol D}}
\def\bL{{\boldsymbol L}}
\def\bA{{\boldsymbol A}}
\def\bM{{\boldsymbol M}}

\def\by{{\boldsymbol y}}
\def\bx{{\boldsymbol x}}
\def\bz{{\boldsymbol z}}
\def\bv{{\boldsymbol v}}

\def\b\Pi{{\boldsymbol \Pi}}

\def\bL{{\boldsymbol L}}

\def\bX{{\boldsymbol X}}
\def\bY{{\boldsymbol Y}}

\def\bE{{\boldsymbol E}}

\def\sTV{{\rm TV}}   
\def\deg{{\rm deg}}

\def\Treg[#1]{T^{{\rm reg},#1}}
\def\GW[#1]{{\rm GW}(#1)}
\def\MGW[#1]{{\rm MGW}(#1)}

\def\toprob{\stackrel{p}{\longrightarrow}}

\def\id{\mathbf{I}}
\def\J{\mathbf{J}}

\title{A Statistical Model for Motifs Detection} 
\author{Hamid Javadi\thanks{Department of Electrical Engineering, Stanford University}\;\;
\; and \;\; Andrea Montanari\thanks{Department of Electrical 
Engineering and Statistics, Stanford University}}

\makeindex

\begin{document}

\maketitle

\begin{abstract}
We consider a statistical model for the problem of finding  subgraphs with specified topology in an otherwise random graph. 
This task plays an important role in the analysis of social and biological
networks. In these types of networks, small subgraphs with a specific structure
have important functional roles, and they are referred to as `motifs.'

Within this model, one or multiple copies of a subgraph is added (`planted')
in an Erd\H{o}s-Renyi random graph with $n$ vertices and edge probability $q_0$. 
We ask whether the resulting graph can be distinguished reliably from
a pure Erd\H{o}s-Renyi random graph, and we present two types of result.
First we investigate the question from a purely statistical
perspective, and ask whether there is \emph{any} test that can
distinguish between the two graph models. We provide necessary and
sufficient conditions that are essentially tight for small enough subgraphs.

Next we study two polynomial-time algorithms for solving the same
problem: a spectral algorithm, and a semidefinite programming (SDP)
relaxation. For the spectral algorithm, we establish sufficient
conditions under which it distinguishes the two graph models with high
probability. Under the same conditions the spectral algorithm indeed identifies the
hidden subgraph.

The spectral algorithm is substantially sub-optimal with respect to
the optimal test. We show that a similar gap is present for the more sophisticated SDP
approach. 
\end{abstract}

\section{Introduction}
\label{sec:intro}

`Motifs' play a key role in the analysis of social and biological
networks. Quoting from an influential paper in this area
\cite{milo2002network}, the term `motif' broadly refers to
 
\begin{quote}
``patterns of interconnections
occurring in complex networks at numbers that are significantly
higher than those in randomized networks.''
\end{quote}

For instance, the authors of \cite{milo2002network} considered directed graph
representations of various types of data: gene regulation
networks, neural circuits, food webs, the world wide web, electronic
circuits. They identified a number of small subgraphs that are found in
atypically large numbers in such networks, and provided interpretations
of their functional role.
The analysis of motifs in large biological networks was pursued in a
number of publications, see e.g. 
\cite{kashtan2004efficient,yeger2004network,kashtan2005spontaneous,song2005highly,alon2007network}.

The analysis of subgraph frequencies has an even longer history within
sociology, in part because sociological theories are predictive of
specific subgraph occurrences. We refer to
\cite{granovetter1973strength} for early insights, and to
\cite{wasserman1994social,easley2010networks} for recent reviews of
this research area. 

This paper studies a statistical model for the motif detection problem.
In order to provide a formal statement of the problem, denote by $\cG_n$ the space of graphs over vertex set $V_n=[n]$.
\begin{definition}
We say that the two (sequences of) probability laws $\prob_{0,n}$, $\prob_{1,n}$ over $\cG_n$ are 
\emph{strongly distinguishable} if there exists a sequence of functions (a `test') $T:\cG_n\to \{0,1\}$ such
that
\begin{align*}
\lim\sup_{n\to\infty}\prob_{0,n}\big(T(G_n) = 1\big) =
\lim\sup_{n\to\infty}\prob_{1,n}\big(T(G_n) = 0\big) = 0\, .
\end{align*}
We say that they are \emph{weakly distinguishable} if there exists
$T:\cG_n\to \{0,1\}$ such that 
\begin{align*}
\lim\sup_{n\to\infty}\Big[\prob_{0,n}\big(T(G_n) = 1\big) +\prob_{1,n}\big(T(G_n) = 0\big) \Big]< 1\, .
\end{align*}
We say that they are \emph{polynomial-time weakly (or strongly) distinguishable} if a test exists that achieves the 
above and can be computed by a polynomial-time algorithm.
\end{definition}

Throughout the paper, $\prob_{0,n}$ will correspond to a standard Erd\H{o}s-Renyi random graph, while $\prob_{1,n}$
will be an Erd\H{o}s-Renyi random graph  with planted copies  of a small graph $H_n$. 
Namely, fix $0\le q_0\le 1$, and let $H_n = (V(H_n),E(H_n))$ be a sequence of graphs, indexed by
$n$. Let us emphasize that $H_n$ is a non-random graph on $v(H_n) \equiv |V(H_n)|\ll n$ vertices.
Given a graph $F$, we denote by $\cL(F,n)$ the set of labelings of
the vertices of $F$ taking values in $[n]$, i.e.
\begin{align}
\cL(F,n) \equiv \Big\{\;\varphi:V(F)\to [n]\;\;\mbox{s.t.}\;\;
\varphi(i) \neq \varphi(j) \;\;\forall i\neq j\;\Big\}\, .
\end{align}
(In particular $| \cL(H_n,n)| = n!/(n-v(H_n))!$.)
For an edge $e=(i,j)\in E(H_n)$, we let $\varphi(e)$ be the unordered pair
$(\varphi(i),\varphi(j))$, and hence $\varphi(E(H_n)) = \{\varphi(e):\;
e\in E(H_n)\}$. 

We then let $\prob_{0,n}((i,j)\in E) = q_0$ independently for all pairs $(i,j)$. On the other hand 
\begin{align}
\prob_{1,n}(\,\cdot\,)& = \frac{1}{|\cL(H_n,n)|}\sum_{\varphi\in\cL(H_n,n)}\prob_{1,n}\big(\,\cdot\,\big|\varphi\big)\, ,\\
\prob_{1,n}\big((i,j)\in E\big|\varphi\big)& = \begin{cases}
1 & \mbox{ if $(i,j)\in \varphi(E(H_n))$,}\\
q_0 & \mbox{ otherwise.}
\end{cases}
\label{equation:alternative}
\end{align}
Here it is understood that edges are independent conditional on the labeling $\varphi$.
In Section \ref{sec:multipleplanted} we will generalize this defintion by considering the case in which a large number of identical
copies of $H_n$ is planted.

Whenever the two laws $\prob_{0,n}$ and $\prob_{1,n}$ are distinguishable, we will also say that the motif $H_n$
is detectable.
We will consider the following fundamental questions:
\begin{itemize}
\item Under which conditions on $q_0$, $\{H_n\}_{n\in\naturals}$ the
  are two laws $\prob_{0,n}$, $\prob_{1,n}$ weakly distinguishable? Under
  which conditions are they strongly distinguishable?
\item Assuming the conditions for distinguishability of $\prob_{0,n}$,
  $\prob_{1,n}$ are satisfied, under which conditions there exists  a
  polynomial-time computable test $T(\,\cdot\,)$ that distinguishes $\prob_{0,n}$ from $\prob_{1,n}$?
\item \emph{What features of the motif  $H_n$ controls it detectability?}
\end{itemize}

The first two questions have attracted  a substantial amount of work since the nineties for 
special cases such as planted cliques or dense subgraphs \cite{jerrum1992large,feige2000finding,feige2001dense}. 
A brief overview of this line of work is presented in Section \ref{sec:Related}. 
However, applied studies investigate a much broader collection of motifs than just cliques 
\cite{milo2002network,alon2007network,wasserman1994social,easley2010networks} and --in fact-- cliques
are rarely the most interesting motif from a scientific perspective. 

It is \emph{a priori} unclear whether intuitions developed on the planted clique problem apply to 
general motif detection. Quantitative implications, namely tight necessary and sufficient conditions for detectability
of an arbitrary motif $H_n$  are even less clear.
The present paper builds upon some key insights that were developed in the earlier literature, to obtain a broader picture 
that is potentially applicable to motifs of scientific interest.

In the rest of the paper, we present the following results:
\begin{enumerate}
\item Section \ref{sec:Statistical} presents upper and lower bounds on the statistical threshold 
for motif detection. The bounds match for sufficiently small subgraphs (in particular for $v(H_n)= n^{o(1)}$) 
yielding a sharp characterization in that regime. The key feature controlling statistical detection turns out to be the maximum 
graph density $d(H_n)$ (see Section \ref{sec:Statistical} for a definition).

Determining sharp detectability conditions for larger motifs, or for background density $q_{0}$ depending on $n$, remains an open problem. 
\item In order to explore the computational limits of motif detection, Section \ref{sec:Spectral} analyzes a simple spectral algorithm
that computes the leading eigenvector of the centered adjacency matrix of $G$. The key feature controlling 
the behavior of this algorithm is the leading eigenvalue  of the adjacency matrix of $H_n$. We prove that
the spectral approach succeeds with high probability if this eigenvalue is larger than  $C\sqrt{n}$.

We also show that, in the same regime, the spectral algorithms can be augmented with a combinatorial step that 
identifies the planted subgraph. We prove that this step is successful under a certain `balancedness' condition on the motif $H_n$.
\item Section \ref{sec:SDP} apply a semidefinite programming (SDP) relaxation of the quadratic
assignment problem  to motif detection. 
We prove that the SDP approach is not successful unless --again-- the leading eigenvalue of $H_n$
is of order $\sqrt{n}$. In other words, a similarly large gap between
statistical and algorithmic thresholds exists
for SDP as for spectral methods.

As is often the case in proving negative results for SDP relaxation, our analysis relies on a careful construction 
of a primal witness. This construction is of potential interest for other applications of quadratic assignment.
\item Finally, Section \ref{sec:multipleplanted} considers the more general case in which $m_n$ copies of the 
motif $H_n$ are planted in the same graph. We obtain upper and lower bounds on the statistical threshold.
These bounds match when the motif size is small enough or the number of copies is small enough. 
\end{enumerate}

These results suggest that motifs detection  might be computationally hard in regimes in
which it is statistically feasible.  As discussed in the next section, this phenomenon has been already extensively investigated for the planted clique problem.
More interestingly, our analysis suggests that the two thresholds are controlled by different features of the motif $H_n$. 
While the statistical threshold depends on the maximum density of $H_n$, the spectral threshold is related to its principal eigenvalue.

\subsection{Related work}
\label{sec:Related}

Statistical models for motif detection have been studied so far only for specific cases.
An important line of work within theoretical computer science has focused on the planted clique problem, 
which corresponds to the case $H_n= K_{k(n)}$, the complete graph over $k(n)$ vertices.  It is a classical result
that the planted model is distinguishable from a pure Erd\H{o}s-Renyi random graph, for $q_0=1/2$, provided $k(n)\ge 2(1+\eps)\log_2 n$
(with $\eps$ arbitrarily small), while it is undistinguishable  for $k(n)\le 2(1-\eps)\log_2 n$, see e.g. \cite{grimmett1975colouring}.
On the other hand,  approximating the size of the largest clique in a graph is hard in a worst case setting, even within a factor 
$n^{1-\eps}$ \cite{hastad1996clique,khot2001improved}.
Starting with Jerrum's seminal work, a number of authors analyzed broad classes of algorithms for the
statistical model. A short  list includes Monte Carlo Markov Chain \cite{jerrum1992large}, spectral algorithms \cite{alon1998finding},
SDP relaxations within the Lovasz-Schrijver hierarchy \cite{feige2000finding}, statistical query models \cite{feldman2013statistical},
message passing algorithms \cite{deshpande2014finding}, and, most recently, SDP relaxations within the Sum-Of-Squares hierarchy
\cite{meka2015sum,deshpande2015improved,raghavendra2015tight,hopkins2015sos,barak2016nearly}.

While these works provide precious insights into the tradeoff between statistical and computational limits for graph estimation,
they cannot be applied directly to  general motif detection. In particular, they do not clarify what features of the motif $H_n$ are relevant 
for detection. 

The planted clique problem is arguably the most studied statistical problem presenting a large gap between optimal statistical estimation
and computationally efficient algorithms. In fact, several works \emph{assume} hardness of planted clique to prove hardness 
of other statistical estimation problems, see e.g. \cite{berthet2013complexity}. Similar in spirit is the recent work in \cite{daniely2014average},
that instead assumes hardness of refuting random satisfiability formulas, as well as \cite{agarwal2011oracle}.

A generalization of the maximum-clique problem is provided by the densest $k$-subgraph problem. Given a graph $G$ 
and an integer $k$, this requires to find a subset of vertices of size $k$ that contains the largest number of edges \cite{feige2001dense}.
The best polynomial-time algorithm guarantees an $O(n^{1/4+\eps})$ approximation ratio \cite{bhaskara2010detecting}.
The recent paper \cite{manurangsi2016almost} proves that an approximation ratio $n^{1/(\log\log n)^{\Omega(1)}}$ cannot be achieved under the exponential 
time hypothesis (see also \cite{feige2002relations,khot2006ruling,alon2011inapproximability} for other inapproximability results). 
Semidefinite programming relaxations were studied, among others in \cite{feige1997densest,bhaskara2012polynomial}.

Once more, these results do not apply directly to the motif detection problem. In particular, it is not clear --in general-- that 
searching for a dense subgraph is necessarily the best approach to detect a given motif $H_n$. 
Further, the focus of this line of work is on worst-case approximation guarantees, rather than on statistical thresholds.

Closer to the scope of the present paper is some of the recent work on `community detection'. A random graph $G=(V,E)$ over $|V|=n$ vertices is generated 
by selecting a subset $S\subseteq V$ of $|S|=k$ vertices uniformly at random. Edges are conditionally independent given $S$. Two vertices
$i,j$ are  connected by an edge with probability $q_1$ if $\{i,j\}\subseteq S$ and probability $q_0<q_1$ otherwise. 
Statistical and computational thresholds for detection and estimation of the set $S$ were studied by a number of authors
\cite{verzelen2015community,arias2014community,hajek2014computational,montanari2015finding,hajek2015information}.
Again, these works focuses uniquely on the density of the planted subgraph, and do not model its structure.

Following the original work on motifs in biological networks \cite{milo2002network}, several papers developed algorithms to sample uniformly
random networks with certain given characteristics, e.g. with a given degree sequence \cite{milo2003uniform,thorne2007generating,blitzstein2011sequential}.
Uniform samples can be used to assess the significance level of specific subgraph counts in a real network under
consideration.
Let, for instance, $N_{H}(G)$ denote the number of copies of a certain small graph $H$ in $G$. If in a real network of interest we find
$N_H(G) = t$, the probability $\prob_{0,n}(N_{H_n}(G_n)\ge t)$ is used as significance level for this discovery.
Let us mention two key differences with respect to problem considered in the present paper.
First, we focus on conditions  under which the laws $\prob_{0,n}$ and $\prob_{1,n}$ are strongly
distinguishable. For instance, in the case of a single planted subgraph, under the null model $N_{H_n}(G_n)=0$ with
  very high probability. Hence, Monte Carlo is not effective in accurately computing $p$-values in this regime.
 
Second, the work of
  \cite{milo2003uniform,thorne2007generating,blitzstein2011sequential}
implicitly assumes that the subgraph $H_n$ has bounded size so that
$N_{H_n}(G_n)$ can be computed in time $n^{v(H_n)}$ by exhaustive search. Here we consider instead large subgraphs  $H_n$, and address the computational
challenge of testing for $H_n$ in polynomial time.
Let us emphasize that, while we typically assume $H_n$ to have
diverging size, in practice it is impossible to perform exhaustive
search already for quite small subgraphs. For instance, if $n=10^5$,
and $v(H_n) = 6$, exhaustive search requires of the order of $10^{30}/6!\ge
10^{27}$ operations.

\subsection{Notations}

Given $n\in\integers$, we let $[n] = \{1,2,\dots,n\}$
denote the set of first $n$ integers.  We write $|S|$ for the cardinality of a
set $S$. We denote by $(n)_k = n!/(n-k)!$ the incomplete factorial.

Throughout this paper, we will use
lowercase boldface (e.g. $\bv = (v_1,\dots,v_n)$, $\bx = (x_1,\dots,x_n)$, etc.) 
to denote vectors and uppercase boldface (e.g. $\bA = (A_{i,j})_{i,j\in[n]}$, $\bY= (Y_{i,j})_{i,j\in[n]}$, etc.)
to denote matrices.
For a vector $\bv\in \reals^n$ and a set $A \subseteq [n]$, 
we define $\bv_A\in \reals^n$ as $(v_A)_i = (v)_i$ for $i \in A$ and $(v_A)_i = 0$,
otherwise. For a matrix $\bY$ we use 
$\left\|\bY\right\|_2$ to denote its spectral norm. Given a square matrix $\bX\in \reals^{n\times n}$, we denote its trace 
by $\Tr(\bX) = \sum_{i = 1}^{n} X_{i,i}$. 
Given a symmetric matrix $\bM$, we denote by $\lambda_1(\bM)\ge
\lambda_2(\bM)\ge \dots\ge \lambda_n(\bM)$ its ordered eigenvalues.

We denote by $\bone_n = (1,1,\dots, 1)\in\reals^n$ the all-ones vector, and by
$\id_{n}$, $\J_{n}=\bone_n\bone_n^{\sT}\in \reals^{n\times n}$ 
the identity and all-ones matrices, respectively. For a matrix $\bA \in \reals^{m\times n}$,
$\rvec(\bA) \in \reals^{mn}$ is the vector whose $l$'th entry is $A_{ij}$ where $i-1$ and $j-1$
are the quotient and the remainder in dividing $l$ by $n$, respectively.
Also, $\bfe_i\in \reals^{n}$ denotes the $i$'th standard unit vector.

A simple graph is a pair $G=(V,E)$, where $V$ is a vertex set and
$E$ is a set of unordered pairs $(i,j)$, $i,j\in E$. We will write
$V(G)$, $E(G)$ whenever necessary to specify which graph we are
referring to. Throughout, we will be focusing on finite graphs.
We let $v(G)= |V(G)|$, $e(G)= |E(G)|$. For $v\in V(G)$, degree of node $v$
is shown by $\deg(v)$.
For $H \subseteq G, v\in V(G)$, the number of nodes in $H$ which are connected to $v$
is denoted by $\deg_{H}(v)$.
We let $\cG_n$ be the set of
graphs over vertex set $V=[n]$.

We follow the standard Big-Oh notation. Given functions $f(n), g(n)$,
we write $f(n)=O(g(n))$ if there exists a constant $C$ such that
$f(n)\le C\, g(n)$, $f(n)=\Omega(g(n))$ if there exists a constant $C$ such that
$f(n)\ge g(n)/C$, and $f(n) = \Theta(g(n))$ if $f(n) = O(g(n))$ and
$f(n) = \Omega(g(n))$. Further $f(n) = o(g(n))$ if $f(n)\le C\, g(n)$
for all  $C>0$ and $n$ large enough, and $f(n) = \omega(g(n))$ if $f(n)\ge C\, g(n)$
for all  $C>0$ and $n$ large enough.


\section{Statistical limits on hypothesis testing} 
\label{sec:Statistical}

In this section we address the first question stated in Section \ref{sec:intro}:
under which conditions on $q_0$, $\{H_n\}_{n\in\naturals}$ are the two
laws $\prob_{0,n}$, $\prob_{1,n}$ distinguishable? We
focus on the case of a single planted subgraph, $m_n=1$, deferring the
generalization to $m_n>1$ to Section \ref{sec:multipleplanted}.
We note that strong and weak distinguishability are equivalent
to $\liminf_{n\to\infty}\|\prob_{0,n}-\prob_{1,n}\|_{\sTV} = 1$, and 
$\liminf_{n\to\infty}\|\prob_{0,n}-\prob_{1,n}\|_{\sTV}>0$, respectively.

Our results depend on the graph sequence $\{H_n\}_{n\in\naturals}$
through its \emph{maximum density} $d(H_n)$.
For a graph $H$, we define
\begin{align}
d(H) \equiv \max_{F\subseteq H} \left(\frac{e(F)}{v(F)}\right)\, .
\end{align}
The following theorem provides a sufficient condition
on the distinguishability of laws $\prob_0$, $\prob_1$.
\begin{theorem}\label{thm:upperbd}
Let $\{H_n\}_{n\ge 1}$ be a sequence of non-empty graphs such that $v(H_n) =
o(n)$ and for $q_0\in (0,1)$ let $\prob_{0,n}$ be the null model
with edge density $q_0$, and $\prob_{1,n}$ be planted model with
parameters $H_n$, $q_0$. If 
\begin{align}
\label{eq:upperbdcond}
\liminf_{n\to\infty} \frac{d(H_n)\log (1/q_0)}{\log n}
  >1\, ,
\end{align}
then the two laws $\prob_{0,n}$, $\prob_{1,n}$ are strongly distinguishable.
\end{theorem}

\begin{remark}
The proof of this theorem also provides an explicit test $T:\cG_n\to
\{0,1\}$ which has asymptotically vanishing error probability under
the assumptions of the theorem. 
Let $k(n) = v(F_n)$, where $F_n\subseteq H_n$ is the subgraph of $H_n$
with the smallest number of vertices, such that $e(F_n)/v(F_n) =
d(H_n)$.  Then, the test developed in the proof requires searching
over all subsets of $k(n)$ vertices which, in most cases, is
non-polynomial. 
\end{remark}

The next theorem provides condition under which the two laws are indistinguishable.
\begin{theorem}\label{thm:lowerbdp1}
Let $\{H_n\}_{n\ge 1}$, $q_0\in (0,1)$, $\prob_{0,n}$, $\prob_{1,n}$
be as in Theorem \ref{thm:upperbd}.  
Then the two models are not 
weakly distinguishable if
\begin{align}
\limsup_{n\to\infty} \frac{d(H_n)\log(1/q_0) + (5/2)\log{v(H_n)}}{\log n}
<1\, .
\end{align}
Further, if $d(H_n) = o(\log{v(H_n)})$,
then the laws $\prob_0$, $\prob_1$ are not weakly distinguishable if
\begin{align}
\limsup_{n\to\infty} \frac{v(H_n)}{n^{1/2}} = 0.
\end{align}
\end{theorem}
Note that, under the condition $v(H_n)=o(n^\alpha)$ for any $\alpha>0$ (i.e. when the hidden
subgraph is `not too large'), or when $d(H_n)/\log v(H_n)\to \infty$ as $n \to \infty$ (i.e. when the hidden
subgraph is `dense enough'),
this bound matches the positive result of Theorem \ref{thm:upperbd}. We illustrate these results with a few examples
in Appendix \ref{section:statisticalexamples}.

\section{Computationally efficient tests}
\label{sec:Algo}

In this section we study two computationally plausible
algorithms for detecting the planted subgraph in the setting described in Section \ref{sec:intro}.
The first method leverages the spectral 
properties of the given graph for solving the problem.
In this case, we establish sufficient conditions under which the
algorithm succeeds with high probability. We then show that a
modification of the  spectral algorithm can be used to identify the hidden subgraph.
The second approach uses an SDP relaxation of the problem, that is
\emph{a priori} more powerful than the spectral approach. 

\subsection{Spectral algorithm}
\label{sec:Spectral}

For $p \in [0,1)$ we denote by $\bA_G^p$ the \emph{shifted adjacency
matrix} of the graph $G$, defined as follows:
\begin{align}\label{equation:encodingadjacency}
\left(\bA_{G}^p\right)_{ij} = \begin{cases}
1 & \mbox{ if $(i,j)\in E(G)$,}\\
-p/(1-p) & \mbox{ otherwise.}
\end{cases}
\end{align}
Further, we will denote by $\bA_G = \bA_G^0$ the $0-1$ adjacency matrix of $G$.
Recall that $\lambda_1(\bA_G^p)\ge \lambda_2(\bA_G^p)\ge \dots\ge
\lambda_n(\bA_G^p)$ denote the eigenvalues of $\bA_G^p$. The spectral test is
simply based on the leading eigenvalue:
\begin{align}
T_{\spec}(G) &= \begin{cases}
1 & \mbox{ if $\lambda_1(\bA_G^{q_0})\ge 2.1\, \sigma(q_0)\sqrt{n}$,}\\
0 & \mbox{otherwise\, ,}
\end{cases}\label{eq:TspecDef}\\
\sigma(q_0)&\equiv \sqrt{\frac{ q_0}{1-q_0}}\, . \label{eq:sigmaDef}
\end{align}
This algorithm was first proposed for the planted clique problem in \cite{alon1998finding}.
Note that this test uses the knowledge of $q_0$, but does not assume the knowledge of
planted subgraph $H_n$.
\begin{theorem}\label{thm:SpectralTest}
Let $\{H_n\}_{n\ge 1}$ be a sequence of non-empty graphs such that $v(H_n) =
o(n)$ and for $q_0\in (0,1)$ let $\prob_{0,n}$ be the null model
with edge density $q_0$, and $\prob_{1,n}$ be planted model with
parameters $H_n$, $q_0$. Define $\sigma(q_0)$ as per
Eq.~(\ref{eq:sigmaDef}). If
\begin{align}
\liminf_{n\to
  \infty}\, \frac{\lambda_{1}(\bA_{H_n})}{\sqrt{n}}> 3\sigma(q_0)\, ,\label{eq:SpectralTestThreshold}
 \end{align}
then the two laws $\prob_{0,n}$, $\prob_{1,n}$ are strongly distinguishable.
\end{theorem}
\begin{remark}
The constant $2.1$ in Eq.~(\ref{eq:TspecDef}) can be
reduced to $2+\eps$ for any $\eps>0$. In addition, we expect that with further work the constant $3$ in
Eq.~(\ref{eq:SpectralTestThreshold})
can be reduced to $1+\eps$ for any $\eps>0$.
These improvements are not the focus of the present paper.
\end{remark}

\begin{algorithm}
\caption{Spectral algorithm for identifying hidden subgraphs in $G$\label{algorithm:spectral}}
\textbf{Input: }Graph $G$, edge probability $q_0$, size  of hidden
subgraph $k=v(H)$\\
\textbf{Output: } Estimated support of the hidden subgraph  $S \subseteq V(G)$\\
\textbf{Initialize: } $n=v(G)$, $t= kq_0 + 3\sqrt{kq_0\log{k}}$,  $S = \varnothing$
\begin{algorithmic}[1]
\FOR{$i\in V(G)$}
\STATE Set $\bv^{(i)}\equiv$ principal eigenvector of
$(\bA_G^{q_0})_{-i,-i}$.
\STATE Order the entries of $\bv^{(i)}$: $|v^{(i)}_{j(1)}|\ge |v^{(i)}_{j(2)}|\ge \dots\ge |v^{(i)}_{j(n)}|$
\STATE  Set $S_i\equiv\{ j(1),\dots,j(k)\}$
\STATE Set $d^{(i)} \equiv $ \# of edges between vertex $i$ and vertices in $S_i$
\IF{$d^{(i)} > t$}
\STATE $S = S\cup\{i\}$
\ENDIF
\ENDFOR
\ENSURE $S$
\end{algorithmic}
\end{algorithm}
Can spectral methods be used to identify the hidden subgraph?
We start by noting that, even if $H_n$ can be detected, a subset of its
node might remain un-identified. As an example, let $H_n$
be a graph over $k(n)$ vertices, whereby vertices $\{1,\dots,k(n)-1\}$
are connected by a clique, and vertex $k(n)$ is connected to the
others by a single edge, see figure below:

\phantom{A}\hspace{4cm}\includegraphics[width=0.28\textwidth, height=0.35\textwidth,angle=-90]{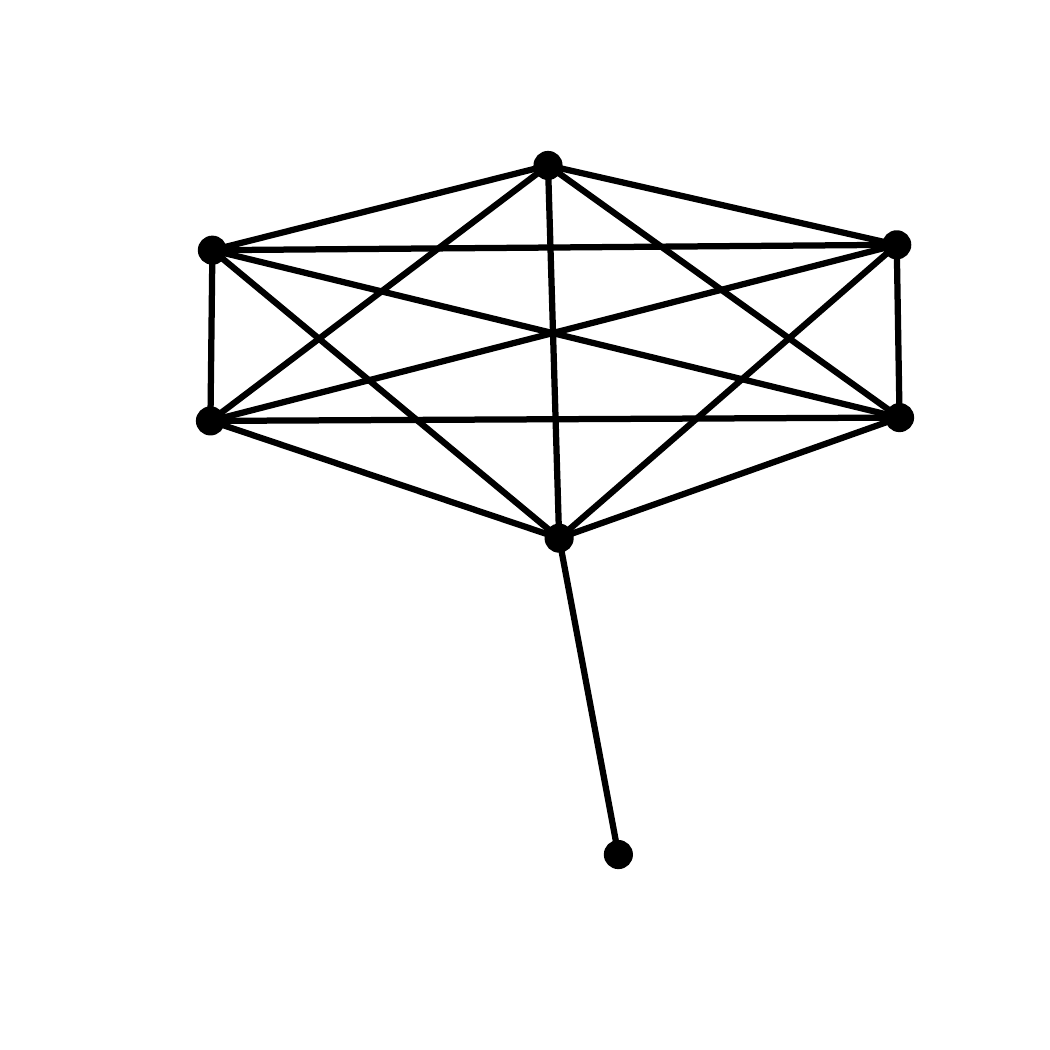}

Then Example \ref{example:FirstClique} implies that $H_n$ can be
detected with high probability as soon as $k(n)\ge (1+\eps) \log
n/\log(1/q_0)$. As we will see below, the spectral algorithm detects
$H_n$ with high probability  if $k(n)\ge 3\sigma(q_0)\sqrt{n} =
\Theta(\sqrt{n})$.
However it is intuitively clear (and not hard to prove) that the
degree-one vertex in $H_n$ cannot be identified reliably.

With this caveat in mind, Algorithm 1  gives a spectral approach to
identify a subset of the vertices of the hidden subgraph.
In order to characterize the set of `important' vertices of $H_n$,
we introduce the following notion.
\begin{definition}\label{def:significantset}
Given a graph $H=(V(H),E(H))$, and $c\in (0,1)$, we define the
$c$-significant set of $H$, 
$S_c(H)\subseteq V(H)$ as the following set of vertices
\begin{align}
S_c(H) := \big\{i \in V(H): \deg(i) > c\, v(H)\big\}\, .
\end{align}
\end{definition}

We also need to to assume that the leading eigenvector of $H$ is
sufficiently spread out.
\begin{definition}\label{def:balancedspectrum}
Let  $H=(V(H),E(H))$ be a graph with adjacency matrix
$\bA_H \in \{0,1\}^{n\times n}$.
For $\eps\in(0,1)$, we say that $H$ has spectral expansion $\eps$, if
\begin{align}
1-\eps \geq \frac{\max\big(\lambda_2(\bA_H);-\lambda_n(\bA_H)\big)}{\lambda_1(\bA_H)}
\end{align}
Finally let $\bv$ be the leading eigenvector of $\bA_H$.
We say that $H$ is \emph{$(\eps,\mu)$-balanced in spectrum} if it has spectral expansion $\eps$ and
\begin{align}
\min_{i \in V(H)} |v_i|\ge\frac{\mu}{\sqrt{v(H)}}.
\end{align}
\end{definition}

The following definition helps us present our result on Algorithm
\ref{algorithm:spectral}.
\begin{definition}
Let $H = \left(V(H),E(H)\right)$ be a  graph. For any $i\in V(H)$, 
the graph obtained by removing  $i$ from $H$ is denoted by $H\setminus
i$. Then:
\begin{enumerate}
\item We say that $H$ is \emph{$(\eps,\mu)$-strictly balanced in spectrum} if 
for all $i \in V(H)$, $H\setminus i$ is $(\eps,\mu)$-balanced in
spectrum.
\item We
define $\lambda_{-}(H)$ as 
\begin{align}
\lambda_{-}(H) \equiv \min_{i\in
  V(H)}\lambda_{1}(\bA_{H\setminus i})
  \, . \label{eq:LambdaMinus}
\end{align}
\end{enumerate}
\end{definition}

The next theorem states sufficient conditions under which Algorithm  \ref{algorithm:spectral}
succeeds in identifying the significant set of the planted subgraph.
\begin{theorem}\label{thm:spectralalgorithm}
Given $\{H_n\}_{n\in \naturals}$, $q_0\in (0,1)$, let  $\prob_{1,n}$
be the law of the random graph with edge density $q_0$ and planted
subgraph $H_n$, cf. Section \ref{sec:intro}, and assume $G_n\sim\prob_{1,n}$. 
Assume $v(H_n) = o(n)$ and that, for each $n$, $H_n$ 
is $(\eps,\mu)$-strictly balanced in spectrum
for some $\mu>0, \eps\in (0,1)$. Let $\delta$ be such that
\begin{align*}
\frac{2\delta}{\mu^2(1-\delta)} < 1.
\end{align*}
Finally, assume that
\begin{align}
 \liminf_{n\to\infty}\frac{|\lambda_{-}(H_n)|}{\sqrt{n}}>\frac{3\sigma(q_0)}{\eps\delta},
 \end{align}
where  $\lambda_{-}(H_n)$ is defined as per
Eq.~(\ref{eq:LambdaMinus}).

Let $S$ be the output of Algorithm \ref{algorithm:spectral}, and set
 $c > \alpha$,
$\alpha \equiv 2\delta/(\mu^2(1-\delta))$.
Then the following statements hold with high probability as $n,v(H_n) \to \infty$:
\begin{enumerate}
\item $S$ contains all the vertices of $G_n$ that correspond to the 
$c$-significant set $S_c(H_n)$ of planted subgraph $H_n$.
\item $S$ does not contain any vertex that does not correspond to those
  of the planted subgraph $H_n$.
\end{enumerate}
\end{theorem}

\begin{remark}
\label{remark:spectralcomplement}
Note that if $\min_{i\in V(H_n)} \deg(i) > cv(H_n)$
where $c$ is as in Theorem \ref{thm:spectralalgorithm}
(the minimum degree of nodes in the hidden subgraph is `sufficiently large'),
$S_c(H_n) = V(H_n)$ and under the assumptions of Theorem \ref{thm:spectralalgorithm}, Algorithm \ref{algorithm:spectral}
will find all the nodes of the planted subgraph $H_n$.

In the opposite case, $H_n$ contains some `low degree' vertices, namely,
for some $i \in V(H_n)$, $\deg(i) \leq cv(H_n)$,
we have $S_c(H_n)\subset V(H_n)$ strictly. Then,
in order to find all vertices of the planted subgraph $H_n$ in $G_n$,
after finding the output of Algorithm \ref{algorithm:spectral}, $S$,
we can select the nodes $i \in V(G_n)$ such that $\deg_{S}(i) > (1+\eps)q_0|S|$, for some $\eps > 0$.
Note that if $|S_c(H_n)|$ is $\omega(\log n)$ then 
for any $i \notin \varphi_0(H_n)$, $\deg_S(i) \leq (1+\eps)q_0|S|$ with high probability.
Hence, this procedure will not choose any node $i$ such that $i \notin \varphi_0(H_n)$.
Moreover, this procedure will find the planted subgraph $H_n$ if for all nodes $i \in V(H_n)$,  
$\deg_{S}(i) > (1+\eps)q_0|S|$.
\end{remark}

Note that for any graph $H = \left(V(H),E(H)\right)$,
$\lambda_1(\bA_H)\le v(H)$.
Hence by definition
\begin{align}
\lambda_{-}(H) \le v\left(H\setminus i\right)  \le\frac{v(H)}{q_0}\, .
\end{align}
Hence,  the assumptions of Theorem \ref{thm:spectralalgorithm} imply
in particular
\begin{align}
\label{eq:spectralnecessary}
\liminf_{n\to\infty}\frac{v(H_n)}{n^{1/2}} > 0\, .
\end{align}
We can compare this condition with the one of Theorem
\ref{thm:upperbd}.
If $H_n$ is a dense graph, we expect generically  $d(H_n)=
\Theta(v(H_n))$, and hence there is a large gap between the condition
of Theorem \ref{thm:upperbd} (that guarantees distinguishability)
and that of Theorem \ref{thm:spectralalgorithm}. We illustrate this
with a few examples in appendix \ref{section:spectralexamples}.

%
%
\subsection{SDP relaxation}
\label{sec:SDP}

Since the  spectral method is generally sub-optimal with respect to the statistical detection threshold,
 it is natural to look for more powerful algorithms.
In this section we use  an SDP relaxation of the quadratic assignment problem, first 
proposed in \cite{zhao1998semidefinite}, for motif detection.

Recall that we denote by $\bA_G$ the adjacency matrix of graph
$G$
\begin{align}
\left(\bA_{G}\right)_{ij} = \begin{cases}
1 & \mbox{ if $(i,j)\in E(G)$,}\\
0 & \mbox{ otherwise.}
\end{cases}
\end{align}
 We want to find a planted copy of a given graph $H$ in 
graph $G$. Let $v(H) = k, v(G) = n$. We consider therefore the problem
\begin{align}\label{equation:QAP}
    \begin{array}{ll}
    \mbox{\maximize}  & \Tr(\bA_H\b\Pi^\sT \bA_G\b\Pi)\\
    \mbox{\subjectto} & \b\Pi^\sT\b\Pi = \id_{k} \\
    & \b\Pi \in \{0,1\}^{n\times k}. \\
    \end{array}
\end{align}
This is a non-convex optimization problem known
as Quadratic Assignment Problem (QAP) and is well studied
in the literature, for example see \cite{burkard2013quadratic}.
We will denote the value of this problem as $\OPT(G;H)$.

Note indeed that, $\b\Pi\in \{0,1\}^{n\times k}$ is feasible if it contains exactly
one non-zero entry per column and at most one per row. Call $\varphi(i)\in [n]$ the position
of the non-zero-entry of column $i\in [k]$. Then $\varphi\in\cL(H,n)$
is a labeling of the vertices of $H$, and the objective function can
be rewritten as
\begin{align}
\Tr(\bA_H\b\Pi^\sT \bA_G\b\Pi) = 2\sum_{(i,j)\in
  E(H)}(\bA_G)_{\varphi(i),\varphi(j)}
\end{align}
Hence, if $G$ contains a planted copy of $H$ (e.g. under model
$G\sim\prob_1$), we have $\OPT(G;H)\ge 2\, e(H)$. This suggests the
following optimization-based test:
\begin{align}
T_{\OPT}(G) = \begin{cases}
1 & \mbox{ if $\OPT(G;H)\ge 2\,e(H)$,}\\
0 & \mbox{ otherwise.}
\end{cases}
\end{align}
The proof of Theorem \ref{thm:upperbd} suggests that this test is
nearly optimal, provided $d(H) = e(H)/v(H)$, i.e. $H$ has no subgraph
denser than $H$ itself\footnote{If this is the not case, the
  optimization problem (\ref{equation:QAP}) can be modified replacing
  $H$ by its densest subgraph.}. 
 
Unfortunately, in general, $\OPT(G;H)$ is NP-complete even to approximate within a
constant factor \cite{sahni1976p}.
We will then resort to an SDP relaxation of the same problem.
The following Lemma provides a different formulation of (\ref{equation:QAP}).
\begin{lemma}
\label{lemma:tensorformulation}
Let $\b\Pi^*$ be an optimal solution of
problem \eqref{equation:QAP}. Then $\rvec(\b\Pi^*) = \by^*$ such that $\by^*\by^{*^\sT} = \bY^*$ is an 
optimal solution of the following problem
\begin{align}\label{equation:tensorformulation}
    \begin{array}{ll}
    \mbox{\maximize}  & \Tr\left((\bA_G\otimes \bA_H)\bY\right) \\
    \mbox{\subjectto} & \bY \in \{0,1\}^{nk\times nk} \\
    & \bY \succeq 0 \\ 
    & \Tr (\bY \J_{nk}) = k^2\\
    & \Tr (\bY (\id_n\otimes(\bfe_i \bfe_i^\sT))) = 1 \;\;\;\;\;\;\;\;\;\; \for \; $i = 1,2,\dots,k$\\
    & \Tr (\bY ((\bfe_j\bfe_j^\sT)\otimes \id_k)) \leq 1 \;\;\;\;\;\;\;\;\;\; \for \; $j = 1,2,\dots,n$\\
    & \rank (\bY) = 1. \\
    \end{array}
\end{align}
\end{lemma}
Now, we try the following 
SDP relaxation of problem \eqref{equation:tensorformulation} which 
is proposed in \cite{zhao1998semidefinite}
\begin{align}\label{equation:SDP}
    \begin{array}{ll}
    \mbox{maximize}  & \Tr\left((\bA_G\otimes \bA_H)\bY\right) \\
    \mbox{subject to} & \bY \succeq 0 \\
    & 0 \le \bY \le 1 \\
    & \Tr (\bY \J_{nk}) = k^2\\
    & \Tr (\bY (\id_n\otimes(\bfe_i \bfe_i^\sT))) = 1 \;\;\;\;\;\;\;\;\;\; \text{for $i = 1,2,\dots,k$}\\
    & \Tr (\bY ((\bfe_j\bfe_j^\sT)\otimes \id_k)) \leq 1 \;\;\;\;\;\;\;\;\;\; \text{for $j = 1,2,\dots,n$}\\
    \end{array}
\end{align}
The following theorem states an upper bound on the
performance of the hypothesis testing 
method that rejects the null hypothesis if
${\sf SDP}(G;H) \geq 2e(H)$.
\begin{theorem}
\label{thm:convexgeneral}
Let $\{H_n\}_{n\ge 1}$, $\prob_{0,n}$, $\prob_{1,n}$
be as in Theorem \ref{thm:upperbd}.
Consider the hypothesis testing problem in which under null $G_n$ is generated according
to $\prob_{0,n}$
and under alternative it 
is generated according to $\prob_{1,n}$.
Define $\sigma(q_0)$ as per
Eq.~(\ref{eq:sigmaDef}). If
\begin{align*}
\limsup_{n\to\infty}\frac{\lambda_{1}(\bA_{H_n})}{\sqrt{n}} < \frac{1}{4}\sigma(q_0),
\end{align*}
then for the method that rejects the null hypothesis if
${\sf SDP}(G_n;H_n) \geq 2e(H_n)$,
\begin{align*}
\P_{0,n}\{T(G_n)=1\} \to 1
\end{align*}
as $n\to \infty$.
\end{theorem}
%
%
\section{The case of multiple planted subgraphs}
\label{sec:multipleplanted}

In this section we would like to generalize the results given in Section \ref{sec:Statistical} to 
the regime in which $m_n>1$ atypical subgraphs are added to a random graph. Namely, fix 
$H_n$ and let $\varphi_1,\varphi_2,\dots,\varphi_{m_n}\in {\cL}(H_n,n)$ be independent uniformly random 
labelings of $V(H_n)$ in $[n]$. As before, we let $\prob_{0,n}$ be the law of an  Erd\H{o}s-Renyi random graph
with edge probability $q_0$. On the other hand under $\prob_{1,n}$ edges are conditionally independent given 
$\varphi_1,\varphi_2,\dots,\varphi_{m_n}$, with
\begin{eqnarray}
\prob_{1,n}\big((i,j)\in E\big|(\varphi_{l})_{l\le m_n}\big) & = & \begin{cases}
1 & \mbox{ if $(i,j)\in \varphi_l(E(H_n))$ for some $l\in \{1,2,\dots,m_n\}$},\\
q_0 & \mbox{ otherwise.}
\label{equation:alternativemultiple}
\end{cases}
\end{eqnarray}
We would like to find the conditions on $m_n, q_0, \{H_n\}$ under which the two laws $\P_{0,n},\P_{1,n}$ 
are strongly or weakly distinguishable, generalizing Theorem \ref{thm:upperbd} and
Theorem \ref{thm:lowerbdp1} to $m_n>1$.  
The following theorem states the sufficient condition under which the two laws 
are indistinguishable.
\begin{theorem}
\label{thm:lowerbdmultiple}
Let $\{H_n\}_{n\geq 1}$ be a sequence of non-empty graphs and for $q_0 \in (0,1)$ let 
$\P_{0,n}$ be the null model with edge density $q_0$ and $\P_{1,n}$ be the planted model as in
\eqref{equation:alternativemultiple} with parameters $H_n$, $q_0$, $m_n$.
Then the two models are not weakly distinguishable 
if
\begin{align}
&\lim\sup_{n\to \infty} \frac{(5/2)\log v(H_n) + \log m_n}{\log n} < 1,
\end{align}
and
\begin{align}
&\lim\sup_{n\to \infty} \frac{d(H_n)\log (1/q_0) + (5/2)\log v(H_n)}{\log n} < 1.\\
\end{align}
\end{theorem}
The following Theorem states  sufficient conditions under which
the two models are strongly distinguishable. 

\begin{theorem}
\label{thm:upperbdmultiplebdd}
Let $\{H_n\}_{n\geq 1}, q_0\in (0,1), m_n, \P_{0,n},\P_{1,n}$ be as in Theorem \ref{thm:lowerbdmultiple}. 
Then the two laws $\P_{0,n}, \P_{1,n}$ are strongly distinguishable if
\begin{align}
\lim \inf_{n\to \infty}\frac{m_n e(H_n)}{n} = \infty\, ,\label{eq:FirstDistinguishable}
\end{align}
or if
\begin{align}
\liminf_{n\to\infty} \frac{d(H_n)\log (1/q_0)}{\log n}
  >1\, .\label{eq:SecondDistinguishable}
\end{align}
\end{theorem}

\begin{remark}
While the necessary and sufficient conditions in the above theorems do
not match in general, they do match in specific regimes of interest. A first
regime is the one of $m_n$ bounded: in this case we recover the same
asymptotics as for $m_n=1$, cf. Section \ref{sec:Statistical}.

A second regime is obtained when the planted graphs $\{H_n\}$
have bounded size: $v(H_n) \leq C$ for all $n$. Theorem
\ref{thm:lowerbdmultiple} implies that 
the two models are not weakly distinguishable if $m_n \le n^{1-\eps}$
for some $\eps>0$ and all $n$ large enough. 

Vice versa, by  Theorem \ref{thm:upperbdmultiplebdd}, they are strongly
distinguishable if $m_n/n\to\infty$. In other words, we have a
characterization of the distinguishability threshold that is tight up
to sub-polynomial factors.
\end{remark}

\begin{remark}
The expected number of edges under the null model $\prob_{0,n}$ is roughly
$n^2q_0/2$, and its standard deviation  is of order $n$. The
sufficient condition (\ref{eq:FirstDistinguishable}) is therefore
equivalent to requiring that the total number of edges in the planted
graphs is much larger than this standard deviation. 
The proof of Theorem  \ref{thm:upperbdmultiplebdd} constructs a simple
test $T:\cG_n\to \{0,1\}$, by letting $T(G_n)=1$ if $e(G_n)\ge t_*$
and $T(G_n)=0$ otherwise. 

Notice that condition (\ref{eq:SecondDistinguishable}) is instead the
same as in Theorem \ref{thm:upperbd}. In this regime, it is sufficient
to find a single copy of the highest density subgraph of $H_n$, and
the multiplicity $m_n$ does not seem to help.
\end{remark}

\begin{remark}
It is possible to use  the spectral algorithm in subsection \ref{sec:Spectral}
to detect the $m_n$ planted subgraphs by setting$k=v(H_n)m_n$ in Algorithm \ref{algorithm:spectral}.

Note that if $m_nv(H_n)/n^{1/2}\to 0$ then
the $m_n$ planted subgraphs will have disjoint vertex sets with high probability as $n\to \infty$. 
Hence the law $\prob_{1,n}$ studied in this section is the same as the one obtained by planting a graph consisting in the disjoint union of $m_n$
copies of $H_n$.
The top eigenvalue of the graph consisting of $m_n$ disjoint copies of a graph $H_n$ is the same as the 
top eigenvalue of $H_n$. Hence, in this regime, we expect the spectral method to succeed in detecting $m_n$
motifs under the same conditions  under which it succeeds in detecting one motif. 
\end{remark}

\section*{Acknowledgments}
H.J. was supported by the William R. Hewlett Stanford Graduate Fellowship.
A.M. was partially supported by NSF grants CCF-1319979 and DMS-1106627 and the AFOSR grant FA9550-13-1-0036.

\bibliographystyle{amsalpha}
\bibliography{all-bibliography}
\addcontentsline{toc}{section}{References}

\appendix

\section{Examples: Statistical limits}
\label{section:statisticalexamples}
\begin{example}\label{example:FirstClique}
Recall that $K_m$ denotes the complete graph over $m$ vertices (hence
having degree $m-1$). Setting $H_n = K_{k(n)}$ we recover the hidden
clique problem. In this case $d(H_n)= (k(n)-1)/2$. Hence, our theorems
imply that the two laws are strongly distinguishable if
$\liminf_{n\to\infty} k(n)/\log n>2/\log(1/q_0)$, and are not weakly
distinguishable if $\limsup_{n\to\infty} k(n)/\log n<2/\log(1/q_0)$.
\end{example}

\begin{example}
\label{ex:hypercube}
Let $Q_m$ be the hypercube graph over $2^m$ vertices (hence having
degree $m$): this is the graph whose vertices are binary vectors of
length $m$, connected by an edge whenever their Hamming distance is
exactly equal to one.
Set $H_n=Q_{\log_2 k(n)}$. In other words, $H_n$ is an hypercube over
$k(n)$ vertices. It is easy to see that $d(H_n) = (\log_2k(n))/2$.

Let $\gamma(q_0) \equiv 2/\log_2(1/q_0)$.
Theorem \ref{thm:upperbd} implies that this graph can be detected
provided $k(n)\ge n^{\gamma(q_0)+\eps}$ for some $\eps>0$ and all $n$
large enough. On the other hand, Theorem
\ref{thm:lowerbdp1} implies that it cannot be detected if
$k(n)\le n^{2\gamma(q_0)/(2+5\gamma(q_0))-\eps}$  for some $\eps>0$ and all $n$
large enough. Hence, we can see that the lower and upper bounds for distinguishability 
are close for small $q_0$ and as $q_0$ increases the gap between the bounds increases.
\end{example}

\begin{example}
\label{ex:regulartree}
Let $H_n$ be a regular tree with degree $d(n)$ and $r(n)$ generations 
(hence $v(H_n) = 1+d(n)[(d(n)-1)^{r(n)}-1]/(d(n)-2)$, $e(H_n) = v(H_n)-1$). In this case
for any $F_n \subseteq H_n$, $e(F_n) \leq v(F_n) - 1$. Therefore,
$d(H_n) = 1-1/v(H_n) < 1$ and Theorem \ref{thm:upperbd}, cannot guarantee the strong distinguishability 
of the hypotheses.
Furthermore, $\lim\sup_{n\to\infty} d(H_n)/\log v(H_n) = 0$
and we are in the low density region. Hence, Theorem \ref{thm:lowerbdp1} implies that
the null and planted models 
are not weakly distinguishable if $\lim\sup_{n\to\infty} d(n)^{r(n)+1}/n^{1/2} = 0$. 
\end{example}

\begin{example}
\label{ex:cycle}
Let $C_{k}^m$ be the $m$-th power of the cycle over $k$ vertices. This
is the graph with vertex set $\{1,\dots,k\}$, and two vertices $i,j$
are connected if $|i-j-bk|\le m$ for some $b\in \naturals$.  Let $H_n=
C_{k(n)}^{m(n)}$, for two functions $m(n)$, $k(n)$. In this case, 
$v(H_n) = k(n)$ and for all $i \in V(H_n)$, $\deg(i) = 2m(n)$. Therefore, for 
any $F_n \subseteq H_n$, $e(F_n)\leq m(n)v(F_n)$. Since $e(H_n) = m(n)k(n)$,
by definition, $d(H_n) = m(n)$. Using Theorem \ref{thm:upperbd},
two models are strongly distinguishable if $\lim\inf_{n\to\infty}m(n)/\log(n) > \log(1/q_0)$.
In addition, depending on $k(n)$, $m(n)$, we can be in different
graph density regimes.
If $m(n) = \omega(\log k(n))$, the laws 
are not weakly distinguishable if $\lim\sup_{n\to\infty}m(n)/\log(n) < \log(1/q_0)$ and we get 
a tight characterization. If $m(n) = o(\log k(n))$,
two models cannot be weakly distinguished if $\lim\sup_{n\to\infty} k(n)/n^{1/2} = 0$. Finally, for the 
intermediate regime where $m(n) = \Theta(\log k(n))$,
if $\lim\sup_{n\to\infty} ((5/2)\log k(n) + m(n)\log(1/q_0))/\log n < 1$ the models are not weakly distinguishable.
 \end{example}

\section{Examples: Spectral Algorithm}
\label{section:spectralexamples}

\begin{example}
\label{ex:spectralclique}
 Let $H_n = K_{k(n)}$. Assume that, $v(H_n) = k(n)$ is $o(n)$. We have $\lambda_1(\bA_{H_n}) = k(n)-1$ and
 Theorem \ref{thm:SpectralTest} implies that the laws are strongly distinguishable
 using the spectral test
 if $\lim\inf_{n\to\infty}k(n)/n^{1/2}>3\sigma(q_0)$. This shows a
 gap between the performance of the spectral test and the statistical bound of Theorem \ref{thm:upperbd}. 
 
 In order to express results on identifying the hidden subgraph in this case, first note that
 for all $i \in V(H_n)$, $\deg(i) = k(n)-1$ and all nodes of $H_n$ are in the $c$-significant
 set of $H_n$ for $c < 1-1/k(n)$ as per Definition \ref{def:significantset}.
 Assuming that $k(n)\to \infty$ as $n\to \infty$, for any $c>0$ all nodes of $H_n$ are in the $c$-significant
 set of $H_n$ for large enough $n$.
 Also, the leading eigenvector of $\bA_{H_n}$ is $\one_{k(n)}$, its corresponding eigenvalue 
 is $k(n)-1$ and the rest of eigenvalues are $-1$.
 Setting $\eps_0(k(n)) \equiv 1-1/(k(n)-1)$, based
 on definition \ref{def:balancedspectrum}, for each $n$, $H_n$ is $(\eps,1)$ balanced 
 in spectrum for $\eps<\eps_0(k(n))$. Using the fact that for any $i \in V(H_n)$, $H_n\setminus i$
 is $K_{k(n)-1}$, we deduce that for each $n$, $H_n$ is $(\eps,1)$- strictly balanced in 
 spectrum for $\eps < \eps_0(k(n)-1)$ and $\lambda_{-}(H_n) = k(n)-2$. Note that,
 $\eps_0(k(n)-1) \to 1$ as $n\to \infty$.
 Therefore, using Theorem \ref{thm:spectralalgorithm}, if $\lim\inf_{n\to\infty}k(n)/n^{1/2} > 9\sigma(q_0)$,
 Algorithm \ref{algorithm:spectral} can find
 the planted clique with high probability as $n, v(H_n) \to \infty$.
\end{example}

\begin{example}
Set $H_n = Q_{\log_2 k(n)}$ as in Example \ref{ex:hypercube}. Since the hypercube is a regular 
graph, $\lambda_1(\bA_{H_n}) = \log_2k(n)$ and Theorem \ref{thm:SpectralTest} implies that 
two models can be strongly distinguished using the spectral test if $\lim\inf_{n\to\infty}\log_2k(n)/n^{1/2}>3\sigma(q_0)$.
However, this never happens since $k \leq n$. Therefore, Theorem \ref{thm:SpectralTest} cannot guarantee the
strong distinguishability of the hypotheses using the spectral test. Similarly, Theorem \ref{thm:spectralalgorithm}
does not imply the success of Algorithm \ref{algorithm:spectral} in finding the planted hypercube.
\end{example}

\begin{example}
Let $H_n$ be a regular tree with degree $d(n)$ and $r(n)$ generations as in Example \ref{ex:regulartree}.
For a large regular tree, $\lambda_1(\bA_{H_n})$ is of order of $2\sqrt{d(n)-1}$ as $v(H_n)\to\infty$.
Hence, based on Theorem \ref{thm:SpectralTest},
two laws are strongly distinguishable using the spectral test if $\lim\inf_{n\to\infty}d(n)/n>(9/4)\sigma^2(q_0)$.
Therefore, $v(H_n)$ cannot be $o(n)$ and Theorem \ref{thm:SpectralTest} cannot
guarantee the strong distinguishability of two models under any conditions.
Recall that Theorem \ref{thm:upperbd}, also, could not guarantee the strong distinguishability 
for this example under any conditions.

As a side note, note that if 
$q_0$ is known a priori -which is not a practical assumption- and $\lim\inf_{n\to\infty}d(n)/(nq_0) > 0 = c$,
the null and planted models can be distinguished only by looking at the maximum degree
in the graph $G_n$. In fact,
under the null model the maximum degree of graph $G_n$
is less than or equal $nq_0 + \Theta(\sqrt{nq_0\log n})$ with high probability. Therefore, the test that rejects the null
iff the maximum degree of $G_n$ is bigger than or equal $(1+\eps)nq_0$ strongly distinguishes two models under this assumption.
Subsequently, under this condition, the high degree nodes can be used to find the planted tree.

In addition, since $d(n)/v(H_n) \to 0$ as $v(H_n) \to \infty$ for the sequence of regular trees,
Theorem \ref{thm:spectralalgorithm} cannot imply the
success of Algorithm \ref{algorithm:spectral} in finding the planted regular tree.
In other words, Algorithm \ref{algorithm:spectral} fails in identifying the planted 
regular tree because it does not contain
\emph{sufficiently high degree} vertices.
\end{example}

\begin{example}
Set $H_n=C_{k(n)}^{m(n)}$ as in Example \ref{ex:cycle}. As we had for previous examples, 
since $H_n$ is a sequence of regular graphs, $\lambda_1(\bA_{H_n}) = 2m(n)$. Therefore, 
two models are strongly distinguishable using the spectral test if $\lim\inf_{n\to\infty}m(n)/n^{1/2}>(3/2)\sigma(q_0)$ 
and the gap between the results of Theorems \ref{thm:upperbd} and \ref{thm:SpectralTest} is similar to
Example \ref{ex:spectralclique}.

Assuming that $\lim\inf_{n\to\infty} m(n)/k(n) = (c/2) > 0$, Theorem
\ref{thm:spectralalgorithm} can used to guarantee the performance of Algorithm \ref{algorithm:spectral} in  
identifying the hidden subgraph. Under this condition, all vertices of $H_n$ are in the 
$c$-significant set of $H_n$ for large enough $n$. Note that $\bA_{H_n}$ is a circulant matrix, its 
principal eigenvalue is $\lambda_1(\bA(H_n)) = 2m(n)$,
corresponding eigenvector is $\one_{k(n)}$ and other eigenvalues are $\sum_{i=1}^{m(n)}2\cos(2\pi ij/k(n))$
for $j = 1,2,\dots,k(n)-1$. Therefore,
$\lambda_2(\bA(H_n)) \leq 2m(n)-\sum_{i=1}^{m(n)}4\pi^2i^2/(k(n))^2 = 2m(n)(1-2\pi^2m(n)^2/(3k(n)^2)) + o(m(n))$. Hence,
for large enough $n$, $\lambda_2(\bA(H_n)) \leq 2m(n)(1-\eps) = \lambda_1(\bA(H_n))(1-\eps)$ and $H_n$
is $(\eps,1)$-balanced in spectrum, where $\eps = \pi^2c^2/6$. For any $i \in V(H_n)$,
$\lambda_1(\bA(H_n\setminus i)) \geq 2e(H_n\setminus i)/(k(n)-1) = 2m(n)(1-1/(k(n)-1))$. 
Using Cauchy's interlacing
theorem, $\lambda_2(\bA(H_n\setminus i))\leq \lambda_2(\bA(H_n))\leq 2m(n)(1-\eps)\leq\lambda_1(\bA(H_n\setminus i))(1-\eps_0)$
where $(1-\eps)/(1-1/(k(n)-1)) = 1-(\eps-\eps^\prime_n) $, $\eps^\prime_n \to 0$ as $n\to \infty$.
In addition, for large enough $n$, for $\bv$, the leading eigenvector of $\bA(H_n\setminus i)$, we have 
$v_i \geq (1-\delta'_n)/\sqrt{k(n)-1}$ for $i = 1,2,\dots,n-1$ where $\delta^\prime_n\to 0$ as $n\to\infty$.
Therefore, for any $0<\eps^\prime<\eps$, $0<\mu<1$,
$H_n$ is $(\eps - \eps^\prime,\mu)$-strictly balanced in spectrum, for large enough $n$.
In addition, $\lambda_{-}(H_n) \geq 2m(n)(1-1/(k(n)-1))$. Thus, using Theorem \ref{thm:spectralalgorithm},
if $\lim\inf_{n\to\infty} m(n)/n^{1/2}>9\sigma(q_0)/(2\eps)$, 
Algorithm \ref{algorithm:spectral}, can find the planted subgraph with high probability as $n, k(n)\to\infty$.

\end{example}

\section{Proofs: Statistical limits}
\label{sec:ProofStatistical}

We start with the following preliminary lemma.
\begin{lemma}\label{lemma:likelihoodtest}
Let, for each $n$, $Z:\mathcal{G}_n\to \reals_+$ be such that, 
\begin{align*}
Z(G_n) = \E_{0,n}\{Z(G_n)\}\frac{\de \prob_{1,n}}{\de\prob_{0,n}}(G_n)\, .
\end{align*} 
Further let $Z_n = Z(G_n)$.
Then, $\prob_{0,n}$ and $\prob_{1,n}$ are strongly distinguishable if and
only if, under $\prob_{0,n}$,
\begin{align*}
\frac{Z_n}{\E_{0,n}Z_n}\toprob 0\, .
\end{align*}
They are not weakly distinguishable if and only if, along some
subsequence $\{n_k\}$,
\begin{align*}
\frac{Z_n}{\E_{0,n}Z_n}\toprob 1\, .
\end{align*}
\end{lemma}
\begin{proof}
First assume that along some subsequence $\{n_k\}$ under $\P_{0,n}$
\begin{align}
\label{assumption1}
\frac{Z_{n}}{\E_0Z_n} \to 1,
\end{align}
in probability. For a test $T: \mathcal{G}_n \to \{0,1\}$, define the risk $\gamma(T)$ as
\begin{align*}
\gamma_{n}(T) = \prob_{0,n}\big(T(G_n) = 1\big) +\prob_{1,n}\big(T(G_n) = 0\big).
\end{align*}
Now, for any test $T$ we have
\begin{align*}
\gamma_n(T) &= \int(1-T)\d\P_{1,n} + \int T\d\P_{0,n}\\
&=\int \left((1-T)\frac{Z_n}{\E_0Z_n} + T \right) \d\P_{0,n}\\
&\ge \int \left(\left(1-\one\left\{\frac{Z_n}{\E_0Z_n}>1\right\}\right)\frac{Z_n}{\E_0Z_n} + \one\left\{\frac{Z_n}{\E_0Z_n}>1\right\}\right)\d\P_{0,n}.
\end{align*}
Using \eqref{assumption1}, the last term goes to $1$ as $n\to \infty$.
Therefore along $\{n_k\}$
\begin{align*}
\liminf_{n\to \infty}\{\inf_{T}\{\gamma_n(T)\}\} \ge 1.
\end{align*}
Which implies that for all tests $T$,
\begin{align*}
\limsup_{n\to \infty}\Big[\prob_{0,n}\big(T(G_n) = 1\big) +\prob_{1,n}\big(T(G_n) = 0\big)\Big] = 1.
\end{align*}
Thus, $\P_{0,n}$, $\P_{1,n}$ are not weakly distinguishable.
\par
Now, assume that
\begin{align}
\label{assumption2}
\frac{Z_{n}}{\E_0Z_n} \to 0.
\end{align}
As above, It is easy to see that the test $T = \one\left\{Z_n/\E_0Z_n>1\right\}$ satisfies
\begin{align*}
\lim\sup_{n\to\infty}\prob_{0,n}\big(T(G_n) = 1\big) =
\lim\sup_{n\to\infty}\prob_{1,n}\big(T(G_n) = 0\big) = 0\, .
\end{align*}
Therefore in this case $\P_{0,n}$, $\P_{1,n}$ are strongly distinguishable.
\end{proof}
In order to state the proof our results,
given a graph $G_n\in\cG_n$, we define $U_{G_n}: \cL(H_n;n)\to \naturals$
by
\begin{align*}
U_{G_n}(\varphi) \equiv \big|\varphi(E(H_n))\cap E(G_n)\big|\, .
\end{align*}
For $n\ge v(H_n)$, we let
\begin{align*}
N(H_n;G_n) \equiv \Big|\big\{\,\varphi\in\cL(H_n;n):\;\; U_{G_n}(\varphi) = e(H_n)\,
\big\}\Big| \, .
\end{align*}
Let $\P_{0,n}, \P_{1,n}$ be defined as in Section \ref{sec:intro}, note that
\begin{align*}
\E_{0,n} N(H_n;G_n) = (n)_{v(H_n)}q_0^{e(H_n)}.
\end{align*}
Further, we can write
\begin{align*}
\frac{\de\prob_{1,n}}{\de\prob_{0,n}}(G_n) = \frac{1}{(n)_{v(H_n)}}\displaystyle\sum_{\varphi\in \cL(H_n;n)}
\displaystyle\prod_{(i,j)\in E(H_n)}
\left[\left (\frac{1}{q_0}\right )
\mathbb I\{\left(\varphi(i),\varphi(j)\right)\in E(G_n)\}\right].
\end{align*}
Thus,
\begin{align*}
\frac{\de\prob_{1,n}}{\de\prob_{0,n}}(G_n) = 
\frac{1}{(n)_{v(H_n)}}\left (\frac{1}{q_0}\right ) ^{e(H_n)}
\displaystyle\sum_{\varphi\in \cL(H_n;n)} \mathbb I\left\{|\varphi(E(H_n)) \cap E(G_n)| = |E(H_n)|\right\}.
\end{align*}
Therefore,
\begin{align}
\label{eq:likelihood}
\frac{\de\prob_{1,n}}{\de\prob_{0,n}}(G_n) = \frac{1}{(n)_{v(H_n)}}\left (\frac{1}{q_0}\right ) ^{e(H_n)}N(H_n;G_n)
= \frac{1}{\E_{n,0}\{N(H_n;G_n)\}}N(H_n;G_n).
\end{align}
Now we can prove Theorems \ref{thm:upperbd}, \ref{thm:lowerbdp1}. 
%
%
\subsection{Proof of Theorem \ref{thm:upperbd}}
\begin{proof}
Let $\tilde H_n$ be a subgraph of $H_n$ that satisfies $d(H_n) = e(\tilde H_n)/v(\tilde H_n)$.
Using \eqref{eq:likelihood}, we can write
\begin{align*}
\P_{0,n}\left(\frac{\de\P_{1,n}}{\de\P_{0,n}}(G_n) > 0\right) &= \P_{0,n}\left(N(H_n;G_n) > 0\right) \\
&\leq \P_{0,n}(N(\tilde H_n;G_n) > 0)\\
&\leq \E_{0,n}N(\tilde H_n;G_n) \\
&\leq n^{v(\tilde H_n)}q_0^{e(\tilde H_n)} \\
&= \exp\left\{v(\tilde H_n)\log n\left(1-\frac{d(H_n)\log(1/q_0)}{\log n}\right)\right\}
\end{align*}
which goes to zero as $n\to \infty$ when \eqref{eq:upperbdcond} holds. Therefore, under the assumptions
of Theorem \ref{thm:upperbd}, under $\P_{0,n}$,
\begin{align*}
\frac{\de\P_{1,n}}{\de\P_{0,n}}(G_n) \toprob 0
\end{align*}
and using Lemma \ref{lemma:likelihoodtest} the proof is complete.
\end{proof}

\subsection{Proof of Theorem \ref{thm:lowerbdp1}}

We start by stating the following lemma.
\begin{lemma}\label{lemma:Nconcentrationp1}
Let $\{H_n\}_{n\ge1}, q_0, \P_{0,n}, \P_{1,n}$ be as in Theorem \ref{thm:upperbd}.
Under the assumptions of Theorem \ref{thm:lowerbdp1}; for all $\eps>0$
\begin{align}\label{Nconcentrationp1}
\lim_{n\to \infty} \P_{0,n}\left\{N(H_n;G_n) \le (1-\eps)\E_0N(H_n;G_n)\right\} = 0.
\end{align}
\end{lemma}
\begin{proof}
Let $v(H_n) = k_n, e(H_n) = e_n$. Let
\begin{align*}
X_{\varphi}(G_n) = \begin{cases}
1 & \mbox{ if $ \left |\varphi(E(H_n))\cap E(G_n)\right | = e_n$,}\\
0 & \mbox{ otherwise.}
\end{cases}
\end{align*}
Note that $N(H_n;G_n) = \displaystyle\sum_{\varphi\in\cL(H_n;n)}X_{\varphi}(G_n)$. 
We have
\begin{align*}
\E_0X_\varphi = p_1(e_n) \equiv q_0^{e_n}.\\
\end{align*}
We write $\varphi_1 \sim \varphi_2$, if $\left|\varphi_1(V(H_n))\cap\varphi_2(V(H_n))\right| \ge 2$. 
We define $e_G(m) = \max_{H \subseteq G, v(H) = \lceil m \rceil} {e(H)}$. Therefore, if 
$\left|\varphi_1(V(H_n))\cap\varphi_2(V(H_n))\right| = u$, we have 
\begin{align*}
\E_0X_{\varphi_1}X_{\varphi_2} \leq p_2(u,e_n) \equiv q_0^{2e_n-e_{H_n}(u)}.
\end{align*}
Define
\begin{align}
\label{eq:deltadefinition}
\bar\Delta(n,H_n) = \sum_{\varphi_1\sim\varphi_2}  \E_0X_{\varphi_1}X_{\varphi_2}.
\end{align}
Therefore, we have
\begin{align*}
\bar \Delta(n,H_n) \le \displaystyle\sum_{u=2}^{k_n} \frac{(n)_{2k_n-u}(k_n!)^2}{u!\left((k_n-u)!\right)^2}p_2(u,e_n).
\end{align*}
 Now using the fact that for all $n$
 \begin{align*}
 \sqrt{2\pi}n^{n+1/2}e^{-n} \le n! \le n^{n+1/2}e^{-n+1},
 \end{align*}
 we get
 \begin{align*}
\frac{\bar\Delta(n,H_n)}{(\E_0 N(H_n;G_n))^2} \le \sum_{u=2}^{k_n} g(u).
 \end{align*}
 Where, 
 \begin{align}
  \label{eq:varoveraverage}
  g(u) = \begin{cases}
 \frac{(2\pi)^{-3/2}e^{-2k_n+2}k_n^{2k_n+1}}
 {(n-u)^u(k_n-u)^{2(k_n-u)+1}e^{-2(k_n-u)-u}u^{u+1/2}}q_{0}^{-e_{H_n}(u)} & \mbox {if $u\leq k_n-1$,} \\
 \frac{k_n^{k_n+1/2}e^{-k_n+1}(n-k_n)^{n-k_n+1/2}e^{-n+k_n+1}}{(2\pi)^{1/2}n^{n+1/2}e^{-n}} q_0^{-e_n} & \mbox {if $u = k_n$}.
 \end{cases}
 \end{align}
 Now using Chebyshev's inequality, for all $\eps > 0$
 \begin{align}
 \label{eq:chebyshev}
 \P_{0,n}\left\{N(H_n;G_n) \le (1-\eps)\E_0N(H_n;G_n)\right\} &\le \frac{\E_0\left(N(H_n; G_n)^2\right)-(\E_0N(H_n; G_n))^2}{\eps^2(\E_0N(H_n; G_n))^2}\nonumber\\
 &= \frac{\sum_{\varphi_1,\varphi_2}\left(\E_0X_{\varphi_1}X_{\varphi_2}-\E_0X_{\varphi_1}\E_0X_{\varphi_2}\right)}{\eps^2(\E_0N(H_n; G_n))^2}\nonumber\\
 &= \frac{\sum_{\varphi_1\sim\varphi_2}\left(\E_0X_{\varphi_1}X_{\varphi_2}-\E_0X_{\varphi_1}\E_0X_{\varphi_2}\right)}{\eps^2(\E_0N(H_n; G_n))^2}\nonumber\\
 &\leq \frac{\sum_{\varphi_1\sim\varphi_2}\E_0X_{\varphi_1}X_{\varphi_2}}{\eps^2(\E_0N(H_n; G_n))^2}=\frac{\bar\Delta(n,H_n)}{\eps^2(\E_0 N(H_n;G_n))^2}.
 \end{align}
Hence, in order to complete the proof it suffices to show that 
 \begin{align}
\frac{\bar\Delta(n,H_n)}{(\E_0 N(H_n;G_n))^2} &\leq \sum_{u=2}^{k_n} g(u)\to 0
\end{align}
as $n\to \infty$.
 Note that for $g(u)$ defined as in \eqref{eq:varoveraverage}, if $u<k_n$
 \begin{align*}
 -\frac{1}{u}\log(g(u)) &\geq \log n + \log(1-(u/n))  - \frac{e_{H_n}(u)}{u}\log(1/q_0)\\
 &\;\;\;\;+ 2(k_n/u)-2((k_n/u)-1)-1-(2(k_n/u)+(1/u))\log k_n\\
 &\;\;\;\;+(2(k_n/u)-2+(1/u))\log(k_n-u)+(1+\frac{1}{2u})\log u\\
 &\geq\log{n}-\frac{e_{H_n}(u)}{u}\log(1/q_0)-2\log{k_n}+\log{u} + \frac{1}{2u}\log{u}\\
 &\;\;\;\;+\big(-2+\frac{1}{u}+\frac{2k_n}{u}\big)\log{\big(1-\frac{u}{k_n}\big)}.
 \end{align*}
 In addition,
  \begin{align*}
 -\frac{1}{k_n}\log(g(k_n)) &\geq (n/k_n+1/(2k_n))\log n- \frac{e_n}{k_n}\log(1/q_0)\\
 &\;\;\;\;-(1+1/(2k_n)) \log k_n - (n/k_n-1+1/(2k_n))\log(n-k_n)+\frac{C}{k_n}\\
 &=\log{n}-\frac{e_n}{k_n}\log(1/q_0)-\log{k_n}-\frac{1}{2k_n}\log{k_n} \\
 &\;\;\;\;-\big(\frac{n}{k_n}-1+\frac{1}{2k_n}\big)\log(1-\frac{k_n}{n})+\frac{C}{k_n}.
 \end{align*}
 Letting
 \begin{align}
 \label{eq:fofu}
 f(u) =  \log{n}-\frac{e_{H_n}(u)}{u}\log(1/q_0)-2\log{k_n}+\log{u} + \frac{1}{2u}\log{u}+\big(-2+\frac{2k_n+1}{u}\big)\log{\big(1-\frac{u}{k_n}\big)},
 \end{align}
 for $2\leq u\leq k_n-1$, and 
 \begin{align}
 \label{eq:fofk}
 f(k_n) =  \log{n}-\frac{e_{H_n}(u)}{u}\log(1/q_0)-\log{k_n} - \frac{1}{2k_n}\log{k_n} -\big(\frac{n}{k_n}-1+\frac{1}{2k_n}\big)\log(1-\frac{k_n}{n}),
 \end{align}
Therefore, it suffices to show that 
 \begin{align}
 \begin{split}
 \label{eq:deltaovern}
\frac{\bar\Delta(n,H_n)}{(\E_0 N(H_n;G_n))^2} &\leq \sum_{u=2}^{k_n} g(u)\\
&\leq C\sum_{u=2}^{k_n} \exp\{-uf(u)\}\\
&\leq Ck_n\exp \left\{-\tilde uf(\tilde u)\right\}\\
&= C\exp\left\{-\tilde u\left(f(\tilde u)-\frac{\log k_n}{\tilde u}\right)\right\} \to 0
\end{split}
 \end{align}
 as $n\to \infty$, where $\tilde u = \arg\min_{2\le u\le k_n}\{uf(u)\}$.
 First note that $((2/x)-2)\log(1-x)\geq -2$, for $0\leq x < 1$. Hence,
 \begin{align*}
 \big(-2+\frac{2k_n}{u}\big)\log{\big(1-\frac{u}{k_n}\big)} \geq -2,
 \end{align*}
 for $2\leq u\leq k_n-1$. Further, since $x\log(1-1/x)$ is increasing for $x>1$, for $2\leq u\leq k_n-1$,
 \begin{align*}
 \frac{1}{u}\log(1-u/k_n) \geq \frac{1}{k_n-1}\log(1-(k_n-1)/k_n) \geq -1,
 \end{align*}
 for large enough $k_n$. In addition, $\log u + (1/(2u))\log u \geq 0$, for $u\geq 1$. Hence, the sum of last three terms in \eqref{eq:fofu} is bounded below.
 Finally, note that 
 \begin{align*}
 &\frac{n}{k_n} -1+\frac{1}{2k_n} \geq 0, \\
 &- \frac{\log k_n}{2k_n} \geq -1,
 \end{align*} 
 for large enough $n$. Thus, the last two terms in \eqref{eq:fofk} are also bounded below.
 Therefore, for $2\leq u \leq k_n$,
 \begin{align}
 \label{eq:fofuminuslogkoveru}
 f(u) - \frac{\log k_n}{u} \geq \log{n}-\frac{e_{H_n}(u)}{u}\log(1/q_0)-(5/2)\log{k_n} + C,
 \end{align}
 for some constant $C$. Hence,
 \begin{equation*}
 \frac{\bar\Delta(n,H_n)}{(\E_0 N(H_n;G_n))^2} \to 0
 \end{equation*}
 as $n\to \infty$ if 
 \begin{align*}
 &\limsup_{n\to\infty} \frac{d(H_n)\log(1/q_0) + (5/2)\log v(H_n)}{\log n}
<1.
 \end{align*}
 This shows 
 that the lemma 
 holds under the assumption of
 Theorem \ref{thm:lowerbdp1}. Now, let $u^* = \displaystyle\arg\min_{2\le u\le k_n} f(u)$. 
 Note that as we had above,
 \begin{align*}
  f(u)  \geq \log{n}-\frac{e_{H_n}(u)}{u}\log(1/q_0)-2\log{k_n} + C.
  \end{align*}
 Hence, under the assumptions of Theorem \ref{thm:lowerbdp1}, $f(u^*) \to \infty$.
 Define
 \begin{align*}
 f^{*}(u) = \begin{cases}
f(u) & \mbox{if, $u \le u^*$,}\\
f(u^*) & \mbox{otherwise}.
\end{cases}
\end{align*}
We have
\begin{align*}
\sum_{u=2}^{k_n}\exp \left\{-uf(u)\right\} &\le \sum_{u=2}^{k_n}\exp \left\{-uf^{*}(u)\right\}\\
&=\sum_{u=2}^{u^{*}}\exp \left\{-uf(u)\right\} + C\exp\{-u^*f(u^*)\}\\
&\le u^*\exp\{-\tilde u f(\tilde u)\} + C\exp\{-u^*f(u^*)\}.
\end{align*}
Where $C = \left(1-e^{-f(u^*)}\right)^{-1}$ is a constant.  
 Therefore, it suffices that 
 \begin{align*}
 u^*\exp\{-\tilde u f(\tilde u)\} + C\exp\{-u^*f(u^*)\}\to 0
 \end{align*}
 as $n \to \infty$.
 This holds when
 \begin{align}
 \label{eq:conditionsonf}
 &u^*f(u^*) \to +\infty,\nonumber\\
 &u^*\exp\{-\tilde u f(\tilde u)\} \to 0.
 \end{align}
 as $n \to \infty$. Note that the first condition above holds since
 under the assumptions of Theorem \ref{thm:lowerbdp1}, $f(u^*)\to\infty$ as $n\to \infty$.
 Further,
 \begin{equation*} 
 \log \left(u^*\exp\{-\tilde u f(\tilde u)\}\right) = \log u^* - \tilde u f(\tilde u).
 \end{equation*}
 Note that if $d(H_n) = o(\log v(H_n)$, then $\limsup_{n\to\infty}(\log u^*)/(\tilde u\log n) = 0$. 
 Thus, \eqref{eq:conditionsonf}
 holds when $n\to \infty$. Therefore, if $d(H_n) = o(\log v(H_n)$ the lemma holds if
  \begin{align}
\label{eq:weakercondition2}
 \limsup_{n\to\infty} \frac{v(H_n)}{n^{1/2}} = 0
 \end{align}
and this completes the proof.
\end{proof}
Now, we can prove Theorem \ref{thm:lowerbdp1}.
\begin{proof}[Proof of Theorem \ref{thm:lowerbdp1}]
using Lemma \ref{lemma:Nconcentrationp1}, under the assumptions of Theorem \ref{thm:lowerbdp1}, for 
all $\eps>0$
\begin{align*}
\lim_{n\to \infty} \P_{0,n}\left\{\frac{1}{\E_0N(H_n;G_n)}N(H_n;G_n) \le (1-\eps)\right\} = 0.
\end{align*}
Therefore, by taking $Z_n = N(H_n;G_n)/\E_0 N(H_n;G_n)$, so that $\E_0 Z_n=1$ we have
\begin{align*}
\E_0[|1 -ˆ' Z_n|] = 2\E_0[(1 -ˆ' Z_n)_+] \to 0
\end{align*}
as $n\to \infty$. 
Therefore, under $\P_{0,n}$
\begin{align*}
\frac{1}{\E_0N(H_n;G_n)}N(H_n;G_n) \toprob 1
\end{align*}
and using \eqref{eq:likelihood} and Lemma \ref{lemma:likelihoodtest}, Theorem \ref{thm:lowerbdp1} is proved.
\end{proof}
%
%
\section{Proofs: spectral algorithm}
We start by stating the following useful theorems from random matrix theory.
\begin{theorem}[\cite{tao2012topics}, Corollary 2.3.6]
\label{lemma:rmt}
Let $\bX \in \reals^{n\times n}$
be a random symmetric matrix whose entries $X_{ij}$ are independent,
zero-mean, uniformly bounded random variables for $j \geq i$ and $X_{ij} = X_{ji}$ for $j < i$.
There exists constants $c_1,c_2 >0$ such that for all $t \geq c_1$
\begin{align*}
\P\left\{\|\bX\|_2 > t\sqrt{n}\right\} \leq c_1 \exp\left(-c_2tn\right).
\end{align*} 
\end{theorem}
\begin{theorem}[\cite{tao2012topics}, Theorem 2.3.24]
\label{lemma:Bai-Yin}
Let $\bX \in \reals^{n\times n}$
be a random symmetric matrix whose entries $X_{ij}$ are i.i.d 
copies of a zero-mean random variable with variance $1$ and finite fourth moment for $j \geq i$ and
$X_{ij} = X_{ji}$ for $j < i$. Then, $\lim_{n\to\infty} \|\bX\|_2/\sqrt{n} = 2$, almost surely.
\end{theorem}
\label{sec:ProofSpectral}
\subsection{Proof of Theorem \ref{thm:SpectralTest}}
First assume that $G_n$ is generated according to the null model $\P_{0,n}$. Then,
$\bA_{G_n}^{q_0}$ is a random symmetric matrix with independent entries where each
entry is a zero-mean Bernoulli random variable which is equal to 
$1$ with probability $q_0$ and $-q_0/(1-q_0)$ with probability $1-q_0$. Using Theorem \ref{lemma:Bai-Yin},
$\lambda_1(\bA_{G_n}^{q_0}) \leq 2.1\sigma(q_0)\sqrt{n}$ with high probability as $n\to \infty$.
Therefore, $\lim\sup_{n\to\infty}\prob_{0,n}\big(T_{\text{spec}}(G_n) = 1\big) = 0$. Now assume that $G_n$
is generated according to the planted model, $\P_{1,n}$, with parameters $q_0$ and $H_n$. Hence, $\bA_{G_n}^{q_0}$
is distributed as $\b\Pi_n^\sT\bA_{H_n}\b\Pi_n + \bE_n$
where $\b\Pi_n \in \{0,1\}^{v(H_n)\times n}$, and $(\b\Pi_{n})_{ij} = 1$ if and only if 
$\varphi_{0,n}(i) = j$. Further, $\bE_n$ is a random symmetric matrix with
independent entries where $(\bE_n)_{i,j} = 0$ if $(\b\Pi_n^\sT\bA_{H_n}\b\Pi_n)_{i,j} = 1$ and
$(\bE_n)_{i,j}$ is a zero mean Bernoulli random variable which is equal to 1
with probability $q_0$ and $-q_0/(1-q_0)$ with probability $1-q_0$, otherwise.
Let $\bv, \|\bv\|_2 = 1$ be the principal 
eigenvector of $\bA_{H_n}$. We have
\begin{align*}
\lambda_1(\bA_{G_n}^{q_0}) 
&\geq \left\langle \b\Pi_n^\sT \bv, \bA_{G_n}^{q_0}\b\Pi_n^\sT \bv\right\rangle
= \left\langle \b\Pi_n^\sT \bv, \b\Pi_n^\sT \bA_{H_n}\b\Pi_n \b\Pi_n^\sT \bv \right\rangle +
\left\langle \b\Pi_n^\sT \bv , \bE_n \b\Pi_n^\sT \bv \right\rangle \\
& = \left\langle \bv, \bA_{H_n}\bv\right\rangle + \left\langle \bv, \b\Pi_n \bE_n \b\Pi_n^\sT \bv\right\rangle
\end{align*}
Therefore,
\begin{align*}
\lim\inf_{n\to\infty} \frac{\lambda_1(\bA_{G_n}^{q_0})}{\sqrt{n}} &\geq 
\lim\inf_{n\to\infty} \frac{\lambda_1(\bA_{H_n})}{\sqrt{n}}
- \lim\sup_{n\to\infty}\frac{\left\langle \bv, \b\Pi_n \bE_n \b\Pi_n^\sT \bv\right\rangle}{\sqrt{n}}\\
& \geq 3\sigma(q_0) - \lim\sup_{n\to\infty}\frac{\lambda_{1}(\b\Pi_n\bE_n\b\Pi_n^\sT)}{\sqrt{n}}.
\end{align*}
Now, using Theorem \ref{lemma:rmt}, 
 $\lambda_{1}(\b\Pi_n\bE_n\b\Pi_n^\sT) \leq c\sqrt{v(H_n)}$, for some $c$, and large enough $n$, almost surely.
 Therefore, $\lim\sup_{n\to\infty}\lambda_{1}(\b\Pi_n\bE_n\b\Pi_n^\sT)/\sqrt{n} = 0$ and under the alternative,
 \begin{align*}
 \lim\inf_{n\to\infty} \frac{\lambda_1(\bA_{G_n}^{q_0})}{\sqrt{n}} \geq 2.1 \sigma(q_0),
\end{align*}
almost surely. Hence, 
$\lim\sup_{n\to\infty}\prob_{1,n}\big(T_{\text{spec}}(G_n) = 0\big) = 0$ and two models are strongly distinguishable 
using the spectral test.

\subsection{Proof of Theorem \ref{thm:spectralalgorithm}}
We start by proving some useful lemmas.
\begin{lemma}\label{lemma:balancedspectrum}
Let $\bA = \b\Pi^\sT \bA_{H_n}\b\Pi + \bE + \tilde\bE$ where $\bA$ is a symmetric $n$ by $n$ matrix,
$\b\Pi \in \{0,1\}^{k \times n}$, $\b\Pi\b\Pi^\sT = \id_k$, $\b\Pi \one=\one$, $\bA_{H_n}$ is $k$ by $k$ symmetric matrix, $k = o(n)$.
Further, let $\bE$ be a random symmetric matrix with independent entries where each
entry is a zero-mean Bernoulli random variable which is equal to 
$1$ with probability $p$ and $-p/(1-p)$ with probability $1-p$.
Finally, $\tilde E_{i,j} = -E_{i,j}$ if $(\b\Pi^\sT \bA_{H_n}\b\Pi)_{i,j} = 1$ and $\tilde E_{i,j} = 0$, otherwise. 
Let $\bv\in \reals^n,\bx\in\reals^k$, $\|\bv\|_2 = \|\bx\|_2 = 1$, be the leading eigenvectors
of $\bA$ and $\bA_{H_n}$, respectively. Assume that for some $\delta\in(0,1)$,
\begin{align*}
\lambda_1(\bA_{H_n})\ge\frac{3}{\eps\delta}\sqrt{\frac{np}{1-p}}
\end{align*}
then $\bv = \alpha\b\Pi^\sT \bx + \bz$ for some $\alpha$ such that 
$\alpha^2\ge 1-\delta$ and $\|\bz\|_2^2 \le \delta$, with high probability as $n\to \infty$.
\end{lemma}
\begin{proof}
Let $S\subseteq[n]$ be the set of $i$'s for which the $i$'th column of $\b\Pi$
is not entirely zero. We denote the complement of this set by $\bar{S}$.
We can write $\bv = \b\Pi^\sT(\alpha\bx+\beta\by) + \bv_{\bar S} = \alpha\b\Pi^\sT\bx + \bz$, where
$\by\in\reals^k$ is such that 
$\by\perp \bx$ and $\alpha^2 + \beta^2 = \|\bv_{S}\|_2^2$. In addition, note that
$\bz \perp \b\Pi^\sT \bx$. Hence, $\|\bz\|^2_2 + \alpha^2\|\b\Pi^\sT \bx\|_2^2 = \|\bv\|_2^2 = 1$. Thus,
$\|\bz\|_2^2 = 1 - \alpha^2\|\b\Pi^\sT \bx\|_2^2 = 1-\alpha^2$.
Now, if $\alpha^2 <1-\delta$, then
\begin{align*}
\left\langle \bv,\bA\bv\right\rangle = \alpha^2\left\langle \bx, \bA_{H_n}\bx \right\rangle + 2\alpha\beta\left\langle \bx,\bA_{{H_n}}\by\right\rangle + \beta^2\left\langle \by,\bA_{H_n} \by\right\rangle + \left\langle \bv,\bE \bv\right\rangle + \langle \bv,\tilde\bE \bv\rangle.
\end{align*}
Since $\bx$ is an eigenvector of $\bA_{H_n}$ and $\bx\perp\by$, we have $\left\langle \bx,\bA_{{H_n}}\by\right\rangle = 0$.
 Now, using Theorems \ref{lemma:rmt}, \ref{lemma:Bai-Yin}, with high probability as $n\to \infty$,
 \begin{align*}
 \langle v,\bA v\rangle&\leq\alpha^2\lambda_1(\bA_{H_n}) + (\|\bv_S\|_2^2 - \alpha^2)(1-\eps)\lambda_1(\bA_{H_n}) + (2+o(1))\sqrt{\frac{np}{1-p}} + c\sqrt{k}\\
 &\leq (1-\delta)\lambda_1(\bA_{H_n}) + \delta(1-\eps)\lambda_1(\bA_{H_n}) + (2+o(1))\sqrt{\frac{np}{1-p}} + c\sqrt{k},
\end{align*}
and
\begin{align}
\langle\b\Pi^\sT\bx, \bA\b\Pi^\sT\bx\rangle = \langle\bx, \b\Pi \bA\b\Pi^\sT\bx\rangle
= \left\langle\bx, \bA_{H_n}\bx\right\rangle + \langle\bx, \b\Pi(\bE+\tilde\bE)\b\Pi^\sT \bx\rangle \geq \lambda_1(\bA_{H_n}) - c^\prime\sqrt{k}.
\end{align}
Therefore, if $\lambda_1(\bA_{H_n}) \geq \frac{3}{\eps\delta}\sqrt{\frac{np}{1-p}}$,
then $\left\langle \bv,\bA\bv\right\rangle<\langle\b\Pi^\sT\bx, \bA\b\Pi^\sT\bx\rangle$ with high probability as $n\to \infty$.
Hence, if $\alpha^2 <1-\delta$,
$\bv$ cannot be the leading eigenvector of $\bA$ and the lemma is proved.  
\end{proof}
The following lemma is an immediate consequence of the above lemma.
\begin{lemma}\label{lemma:spectralalg}
Let $\{H_n\}_{n\ge 1}$ be a sequence of graphs that are $(\eps,\mu)$-balanced in spectrum
for some $\mu>0, \eps\in (0,1)$. Further, let
$\varphi_{0,n}\in\cL(H_n,n)$ be a labeling of $H_n$ vertices in $[n]$, $v(H_n) = o(n)$.
Suppose that $G_n$ is generated according to $\prob_{1,n}(\,\cdot\, |\varphi=\varphi_0)$ as in Eq.~\eqref{equation:alternative}.
Take $\bv$ to be the leading eigenvector of $\bA_{G_n}^{q_0}$.
Let $|v_{j(1)}|\geq|v_{j(2)}|\geq\dots\geq|v_{j(n)}|$ be the entries of $\bv$ and $S^\prime=\{j(1),j(2),\dots,j(v(H_n))\}$.
If
\begin{align*} 
\lambda_1(\bA_{H_n})\ge\frac{3}{\eps\delta}\sqrt{\frac{nq_0}{1-q_0}},
\end{align*}
then
\begin{align*} 
|S^\prime\cap \varphi_0(V(H_n))|\ge \left(1-\frac{2\delta}{\mu^2(1-\delta)}\right)v(H_n),
\end{align*}
with high probability
as $n\to \infty$.
\end{lemma}

\begin{proof}
Note that $\bA_{G_n}^{q_0}$ is distributed as $\b\Pi_n^\sT\bA_{H_n}\b\Pi_n + \bE_n + \tilde\bE_{n}$
where $\b\Pi_n \in \{0,1\}^{v(H_n)\times n}$, and $(\b\Pi_{n})_{ij} = 1$ if and only if 
$\varphi_{0,n}(i) = j$. Further, $\bE_n$ is a random symmetric matrix with
independent entries where each entry is a zero-mean Bernoulli random variable which is equal to 1
with probability $q_0$ and $-q_0/(1-q_0)$ with probability $1-q_0$.
Finally, $(\tilde\bE_n)_{i,j} = -(\bE_n)_{i,j}$ if $(\b\Pi^\sT \bA_{H_n}\b\Pi)_{i,j} = 1$ and $(\tilde\bE_n)_{i,j} = 0$, otherwise.
Hence, defining $\bx$ to be the leading eigenvector of $\bA_{H_n}$, using Lemma \ref{lemma:balancedspectrum},
$\bv = \tilde{\bx} + \bz$, $\tilde{\bx} = \b\Pi_n^\sT\bx$,
$\|\bx\|^2 \geq 1-\delta$ and $\bz \perp \tilde{\bx}$, with high probability.
Let $S = \varphi_0(V(H_n))$, using the assumption that $H_n$ is $(\eps,\mu)$-balanced in spectrum, 
for $i \in S$, $|\tilde x_i|\geq \mu\sqrt{1-\delta}/\sqrt{v(H_n)}$. Note that for $i \notin S$, $\tilde x_i = 0$.
Therefore, for any $i \in (\bar S \cap S^\prime)$, there exists an index $i^\prime\in(\bar{S^\prime}\cap S)$ such that
$z_i^2 + z_{i^\prime}^2 \geq 2\left(\mu\sqrt{1-\delta}/(2\sqrt{v(H_n)})\right)^2$. Hence,
letting $N$ be the number of indices in $S^\prime$ which are not in $S$, we have
\begin{align*}
2N\left(\frac{\mu\sqrt{1-\delta}}{2\sqrt{v(H_n)}}\right)^2 \leq \|\bz\|_2^2 = 1-\|\bx\|_2^2 \leq \delta.
\end{align*}
Therefore, $N \leq 2\delta v(H_n)/(\mu^2(1-\delta))$ and 
$|S^\prime\cap \varphi_0(V(H_n))|\ge \left(1-\frac{2\delta}{\mu^2(1-\delta)}\right)v(H_n)$ with high probability
as $n\to \infty$.
\end{proof}
Now we prove Theorem \ref{thm:spectralalgorithm}. 
\begin{proof}[Proof of Theorem \ref{thm:spectralalgorithm}]
First assume that $i \notin \varphi_0(V(H_n))$. Recall that $d^{(i)}$ is the number of edges 
between vertex $i$ and vertices in $S_i$. Further $S_i$ only depends on the edges induced by $V\setminus \{i\}$. Hence, we have
\begin{align*}
d^{(i)} = \sum_{j = 1}^{v(H_n)} X_j
\end{align*}
where $\{X_j\}$ is a sequence of i.i.d Bern($q_0$) random variables. Therefore using Bernstein's inequality
\begin{align*}
\sum_{i \notin \varphi_0(V(H_n))}\P \left\{d^{(i)} > t\right\} &\leq n\P \left\{d^{(i)} > t\right\} \\
&\leq n\exp\left\{-\frac{(1/2)(t-v(H_n)q_0)^2}{(1/3)(t-v(H_n)q_0) + v(H_n)q_0(1-q_0)}\right\} \\
&\leq n\exp\left\{-\frac{(9/2)v(H_n)q_0\log v(H_n)}{\sqrt{v(H_n)q_0\log v(H_n)}+v(H_n)q_0(1-q_0)}\right\}\\
&\leq \exp\left\{\log n - 4\log v(H_n)\right\},
\end{align*}
for large enough $n$. Using \eqref{eq:spectralnecessary}, this goes to zero as $n\to\infty$. Therefore, using union bound,
the output set $S$ of Algorithm
\ref {algorithm:spectral} will be a subset of $\varphi_0(V(H_n))$.

Next, assume that $i = \varphi_0(\tilde i)$ and $\tilde i \in S_c(H_n)$ where $c$ is as in Theorem \ref{thm:spectralalgorithm}.
Using Lemma \ref{lemma:spectralalg},
$|S_i \cap \varphi_0(V(H_n))| \ge (1-\alpha)v(H_n)$.
Also, note that $S_i$ only depends on the edges induced by $V\setminus \{i\}$.
Therefore, $d^{(i)}$ dominates $d^\prime$, where
\begin{align*}
d^\prime = (c-\alpha)v(H_n) + \sum_{j=1}^{(1-c+\alpha)v(H_n)}X_j,
\end{align*}
in which $\{X_j\}$ is a sequence of i.i.d Bern($q_0$) random variables.
Note that here $\E d^\prime = v(H_n)q_0 + (1-q_0)(c-\alpha)v(H_n)$. Hence, using Bernstein's inequality we get
\begin{align*}
\sum_{i\in \varphi_0(S_c(H_n))}&\P \left\{d^{(i)} \le t \right\} \le v(H_n) \P \left\{d^{\prime} \le t \right\}\\
&\leq v(H_n)\exp\left\{-\frac{(1/2)(v(H_n)(q_0+(1-q_0)(c-\alpha))-t)^2}{(1/3)t + q_0(1-q_0)(1-c+\alpha)v(H_n)} \right\}\\
&\leq v(H_n)\exp\left\{-\frac{(1/4)(1-q_0)^2(c-\alpha)^2v(H_n)^2}{(1/3)(1-c+\alpha)v(H_n)+q_0(1-q_0)(1-c+\alpha)v(H_n)}\right\}\\
&\leq v(H_n)\exp\left\{-C^\prime v(H_n)\right\} \to 0,
\end{align*}
as $n, v(H_n)\to \infty$. Thus, by union bound, the output of Algorithm \ref{algorithm:spectral},
contains all the nodes in the $c$-significant
set of planted subgraph $H_n$ in $G_n$ and has no nodes which are not in 
the planted subgraph $H_n$, with high probability, as $n,v(H_n)\to \infty$.
\end{proof}
%
%
\section{Proofs: SDP relaxation}
\label{sec:ProofSDP}

For simplicity we denote $v(H_n)$ by $k_n$.
We start by proving Lemma \ref{lemma:tensorformulation}.
\subsection{Proof of Lemma \ref{lemma:tensorformulation}}
First note that
every feasible $\b\Pi$ in \eqref{equation:QAP}, corresponds uniquely to an injective
mapping $\varphi$ from $[k]$ to $[n]$ where $\varphi(i) = j$ if and only if $\Pi_{ji} = 1$. Based on this,
we have
\begin{align*}
\Tr\left(\bA_H\b\Pi^\sT \bA_G\b\Pi\right) = \sum_{i,j = 1}^{k}\left(\bA_H\right)_{ij}\left(\bA_G\right)_{\varphi(i)\varphi(j)} = \Tr\left((\bA_G\otimes \bA_H)\bY\right).
\end{align*}
where $\bY = \by\by^\sT$ and $\by = \rvec(\b\Pi)$. Moreover $\bY$ is a 
rank one positive definite matrix in $\{0,1\}^{nk\times nk}$. Also,
\begin{align}
\label{equation:sumallY}
\Tr (\bY \J_{nk}) = \sum_{i,j = 1}^{nk} Y_{ij} = \left(\sum_{i = 1}^{nk}y_i\right)^2 = k^2.
\end{align}
In addition, for $i = 1,2,\dots,k$
\begin{align}
\label{equation:sumrowsY}
\Tr(\bY(\id_n\otimes(\bfe_i\bfe_i^\sT))) = \sum_{l=0}^{n-1}y_{kl+i} = \sum_{j = 1}^{n}\Pi_{ji} = 1.
\end{align}
Further, for $j = 1,2,\dots,n$
\begin{align}
\label{equation:sumcolsY}
\Tr (\bY ((\bfe_j\bfe_j^\sT)\otimes \id_k)) = \sum_{l=0}^{k-1}y_{(j-1)k+l} = \sum_{i = 1}^{k}\Pi_{ji} \leq 1.
\end{align}
Therefore, $\bY$ is feasible for problem \eqref{equation:SDP}. Conversely, if $\bY$ is feasible for 
problem \eqref{equation:SDP}, then $\bY = \by\by^\sT$ where $\by \in \{0,1\}^{nk}$. Also, 
using \eqref{equation:sumallY}, 
$\by$ has exactly $k$ entries equal to one and $n-k$ entries equal to zero. Further, using the first
equality in \eqref{equation:sumrowsY}, we deduce that for $i = 1,2,\dots,k,$
\begin{align*}
\sum_{l=0}^{n-1}y_{kl+i} = 1.
\end{align*} 
Also, using the first equality in \eqref{equation:sumcolsY} for $j = 1,2,\dots,n,$
\begin{align*}
\sum_{l=0}^{k-1}y_{(j-1)k+l} \leq 1.
\end{align*} 
This means that the matrix $\b\Pi \in \{0,1\}^{n\times k}$ whose 
$j$'th row is $\left[y_{(j-1)k+1}, y_{(j-1)k+2}, \dots, y_{jk}\right]$ has exactly one entry equal
to one in each column. Therefore, $\b\Pi$ is feasible for problem equation \eqref{equation:QAP}.
This completes the proof of Lemma \ref{lemma:tensorformulation}.
\subsection{Proof of Theorem \ref{thm:convexgeneral}}
The following lemma about the spectrum 
of a random Erd\H{o}s-R\'enyi graph is a consequence of Theorem \ref{lemma:Bai-Yin}.
\begin{lemma}\label{lemma:randomspectrum}
Let $\bA \in \{0,1\}^{n\times n}$ be a random matrix with independent entries such that
\begin{align*}
A_{ij} = \begin{cases}
1 & \mbox{ with probability $p_{ij}$}\\
0 & \mbox{ with probability $1-p_{ij}$,}
\end{cases}
\end{align*}
where $p_{ij} = p$ if $i\neq j$, $p_{ii} = 0$ and $p\in (0,1)$ is a constant. Then,
\begin{align*}
&\lim_{n\to \infty}\frac{\lambda_1(\bA)}{np} = 1,\\
&\lim\sup_{n\to \infty}\frac{-\lambda_n(\bA)}{2\sqrt{np(1-p)}} = 1,
\end{align*}
almost surely.
\end{lemma}
\begin{lemma}\label{lemma:randomspectrumlaplace}
Let $\bA \in \{0,1\}^{n\times n}$ be a random matrix with independent entries such that
\begin{align*}
A_{ij} = \begin{cases}
1 & \mbox{ with probability $p_{ij}$}\\
0 & \mbox{ with probability $1-p_{ij}$,}
\end{cases}
\end{align*}
where $p_{ij} = p$ if $i\neq j$, $p_{ii} = 0$ and $p \in (0,1)$ is a constant.
Let $\bD$ be a $n$ by $n$
diagonal matrix such that $D_{ii} = \displaystyle\sum_{j=1}^{n}A_{ij}$, $\bL = \bD-\bA$.
Then,
\begin{enumerate}[(i)]
\item $\bL \succeq 0$.
\item $\bL \one_{n} = 0$.
\item if $\lambda_2(\bL)$ is the second smallest eigenvalue of $\bL$ then,
\begin{align*}
\lambda_2(\bL) = np - 2(1+o(1))\sqrt{np(1-p)} 
\end{align*}
almost surely, as $n \to \infty$.
\end{enumerate}
\end{lemma}
\begin{proof}
The proof of points (i), (ii) in Lemma \ref{lemma:randomspectrumlaplace} is standard and 
can be found in \cite{chung1997spectral}. Further, we can write $\bL = \E\bL + (\bD - \E\bD) - (\bA - \E\bA)$
where $\E\bL = np\id_n - p\J_n$ and $\lambda_2(\E\bL) = \lambda_3(\E\bL) = \dots = \lambda_n(\E\bL) = np$.
Therefore, point (iii) follows from this together with the bound on the spectra of $\bA - \E\bA$, $\bD - \E\bD$. 
\end{proof}
Now we can state the proof of Theorem \ref{thm:convexgeneral}.
\begin{proof}[Proof of Theorem \ref{thm:convexgeneral}]
Using the fact that for any graph $G$ with adjacency matrix
$\bA_G \in \{0,1\}^{n\times n}$, $\lambda_1(\bA_G)\ge -\lambda_n(\bA_G)$, it suffices to
prove the theorem assuming that 
\begin{align*}
\limsup_{n\to \infty} {\frac{2(\lambda_{1}(\bA_{H_n})-\lambda_{k_n}(\bA_{H_n}))\sqrt{1-q_0}}{\sqrt{nq_0}}} = 1- C < 1.
\end{align*}
Assume that $G_n$ is generated
randomly according to $\P_{0,n}$.
Let ${\sf SDP}(G_n;H_n)$ be the sequence of the
optimal values of the (random) convex programs \eqref{equation:SDP}.
Let $\bD_{G_n}$ be a $n\times n$
diagonal matrix such that $(\bD_{G_n})_{ii} = \deg(i)$. In order to prove Theorem \ref{thm:convexgeneral},
we have to show that ${\sf SDP}(G_n;H_n)\geq 2e(H_n)$. In order to do this, we construct a sequence of matrices
$\bY_n$ which are feasible for problem \eqref{equation:SDP} and $\Tr\left((\bA_{G_n}\otimes \bA_{H_n})\bY_n\right) \ge 2e(H_n)$,
with high probability as $n \to \infty$.
We take
\begin{align}
\label{equation:constructedY}
\bY_n = \begin{cases}
a_n(\bD_{G_n}\otimes \id_{k_n})+ b_n(\bA_{G_n}\otimes (\bA_{H_n}+\id_{k_n})) + c_n\J_{nk_n} & \mbox{If $\lambda_{k_n}(\bA_{H_n})\ge(2e(H_n)-k_n^2)/k_n$},\\
b_n(\bD_{G_n}-\bA_{G_n})\otimes \id_{k_n} + b_n \bA_{G_n}\otimes \J_{k_n} & \mbox{otherwise},
\end{cases}
\end{align}
where
\begin{align*}
u_n &= \left[-\lambda_{k_n}(\bA_{H_n})-1+
\frac{k_n\lambda_{k_n}(\bA_{H_n})+k_n^2-2e(H_n)}{nk_n^2}\right]_{+},\\
a_n &= \frac{2e(H_n)+k_nu_n+nk_n^2-k_n^2}{2k_ne(G_n)(nk_n-1)},\\
b_n &= \frac{1}{2e(G_n)},\\
c_n &= \frac{k_n(k_n-1)-2e(H_n)-k_nu_n}{n^2k_n^2-nk_n}.
\end{align*}
Now, we show that $\bY_n$ is feasible for problem \eqref{equation:SDP}. 
First, consider the case where $\lambda_{k_n}(\bA_{H_n})\ge(2e(H_n)-k_n^2)/k_n$. In this case,
\begin{equation}
\label{eq:ulessthan}
u_{n}\le k_n-1-2e(H_n)/k_n \leq k_n.
\end{equation}
Hence, $c_n \ge 0$.
Also, $a_n\ge 0$ and $b_n \ge 0$. Thus,
$\bY_n \ge 0$, entrywise.
In addition, 
\begin{align}
\label{eq:maxdeg}
&\max_{i \in V(G_n)}\deg(i) < 2nq_0,\\
\label{eq:edges}
&e(G_n) > q_0n^2/4,
\end{align}
with high probability as $n\to \infty$. Thus,
using the fact that, 
$2e(H_n)\le k_n^2$, $n+1\le2n$ and for large enough $n$,
$nk_n-1\ge nk_n/2$, $n^2k_n^2-nk_n\ge n^2k_n^2/2$,
\begin{align}
\label{equation:elements1}
a_n\left(\max_{i \in V(G_n)}\deg(i)\right)+c_n &\leq \frac{2nq_0(nk_n^2 + k_nu_n)}{(n^2/2)k_nq_0(nk_n/2)} + \frac{(k_n^2-k_n)-k_nu_n}{n^2k_n^2/2}\nonumber\\
&\leq\frac{2n^2k_n^2q_0}{n^3k_n^2q_0/4} + \frac{2nq_0k_nu_n}{n^3k_n^2q_0/4} + \frac{k_n^2}{n^2k_n^2/2} \nonumber\\
&= \frac{16}{n} + \frac{8u_n}{n^2k_n} + \frac{2}{n^2},
\end{align}
which is less than $1$ for large enough $n$.
Also, similarly, using \eqref{eq:edges},
\begin{align}
b_n+c_n &\le \frac{1}{(q_0n^2)/2} + \frac{k_n^2}{n^2k_n^2/2} \nonumber\\
\label{equation:elements2}
&= \frac{2}{q_0n^2} + \frac{2}{n^2} \le 1,
\end{align}
for large enough $n$. Finally, using \eqref{eq:ulessthan}, \eqref{eq:edges},
\begin{align}
u_nb_n+c_n &\le \frac{k_n}{(q_0n^2)/2} + \frac{k_n^2}{n^2k_n^2/2}\nonumber\\
\label{equation:elements3}
&= \frac{2k_n}{q_0n^2} + \frac{2}{n^2} \le 1
\end{align}
for large enough $n$. Therefore, according to the construction of $\bY_n$ as in \eqref{equation:constructedY},
using equations \eqref{equation:elements1},\eqref{equation:elements2},\eqref{equation:elements3} for large enough $n$,
$\bY_n \le 1$, entrywise.
Also,
\begin{align*}
\Tr (\bY_n \J_{nk_n}) = 2a_ne(G_n)k_n + 4b_ne(H_n)e(G_n) + 2b_nk_nu_ne(G_n)+c_nn^2k^2_n = k_n^2.
\end{align*}
Moreover, for $i = 1,2,\dots,k_n$,
\begin{align*}
\Tr (\bY_n (\id_n\otimes(\bfe_i \bfe_i^\sT))) = 2e(G_n)a_n + nc_n = 1.
\end{align*}
Finally, for $j = 1,2,\dots,n$,
\begin{align}
\Tr (\bY_n ((\bfe_j \bfe_j^\sT)\otimes \id_{k_n})) &\leq k_na_n\left(\max_{i \in V(G_n)}\deg(i)\right)+k_nc_n\nonumber\\
\label{equation:elements4}
&\leq \frac{16k_n}{n} + \frac{8u_n}{n^2} + \frac{2k_n}{n^2} \leq 1,
\end{align}
for large enough $n$. Second inequality in \eqref{equation:elements4} is by \eqref{equation:elements1}.
Next, we consider the case in which $\lambda_{k_n}(\bA_{H_n})\ge(2e(H_n)-k_n^2)/k_n$. In this case, since
\begin{align*}
\frac{\max_{i \in V(G_n)}\deg(i)}{2e(G_n)} \leq \frac{2nq_0}{(n^2q_0/2)} \leq 1
\end{align*}
with high probability, as $n\to \infty$,
$0\le \bY_n\le1$ entrywise. Further, for $i = 1,2,\dots,k_n$ and $j = 1,2,\dots,n$,
\begin{align*}
&\Tr (\bY_n \J_{nk_n}) = 2b_nk_n^2e(G_n) = k_n^2,\\
&\Tr (\bY (\id_n\otimes(\bfe_i \bfe_i^\sT))) = 2e(G_n)b_n = 1,\\
&\Tr (\bY_n ((\bfe_j \bfe_j^\sT)\otimes \id_{k_n})) \leq k_nb_n\left(\max_{i \in V(G_n)}\deg(i)\right)\leq \frac{2k_nnq_0}{(n^2q_0/2)} = \frac{4k_n}{n} \leq 1,
\end{align*}
for large enough $n$, with high probability. 
Finally, we have to show that the proposed $\bY_n$ is positive semidefinite with high probability.
In order to show this it is sufficient to show that
\begin{align*}
\tilde \bY_n = 2e(G_n)\tilde a_n(\bD_{G_n}-\bA_{G_n})\otimes \id_{k_n} + \bA_{G_n}\otimes \tilde \bA_{H_n} \succeq 0,
\end{align*}
where
\begin{align*}
\tilde \bA_{H_n} =
\begin{cases}
\bA_{H_n} + (u_n+2e(G_n)a_n)\id_{k_n},& \mbox{ if $\lambda_{k_n}(\bA_{H_n}) \ge (2e(H_n)-k_n^2)/k_n$,}\\
\J_{k_n}& \mbox{ otherwise,}
\end{cases}
\end{align*}
and
\begin{align*}
\tilde a_n =
\begin{cases}
a_n,& \mbox{ if $\lambda_{k_n}(\bA_{H_n}) \ge (2e(H_n)-k_n^2)/k_n$,}\\
b_n& \mbox{ otherwise.}
\end{cases}
\end{align*}
If $\lambda_{k_n}(\bA_{H_n}) < (2e(H_n)-k_n^2)/k_n$, then
$\tilde \bA_{H_n} = \J_{k_n} \succeq 0$. Otherwise,
\begin{align*}
\lambda_{k_n}(\tilde \bA_{H_n}) &= \lambda_{k_n}(\bA_{H_n}) + u_n + 2a_ne(G_n) \\
&\ge \lambda_{k_n}(\tilde \bA_{H_n}) - \lambda_{k_n}(\tilde \bA_{H_n}) - 1 + 
\frac{k_n\lambda_{k_n}(\bA_{H_n})+k_n^2-2e(H_n)}{nk_n^2}\\
& \;\;\;\;+ \frac{2e(H_n)+k_nu_n+nk_n^2-k_n^2}{k_n(nk_n-1)}\\
& = \frac{1}{nk_n-1}\left(u+\lambda_{k_n}(\tilde \bA_{H_n})+1 - \frac{k\lambda_{k_n}(\tilde \bA_{H_n})+k_n^2-2e(H_n)}{nk_n^2}\right) \ge 0.
\end{align*}
Therefore, $\tilde \bA_{H_n}$ is positive semidefinite in both cases.\\
Let ${\bz}$ be an arbitrary vector in $\reals^{nk_n}$. We can write
\begin{align*}
\bz = \bz_{\parallel} + \bz_{\perp} 
\end{align*}
where
\begin{align*}
\bz_{\parallel} &= \one_n\otimes \bx,\;\;\;\;\;\;\;\;\;\bx\in\reals^k,\\
\bz_{\perp} &= \begin{bmatrix}
    \by_{1} \\
    \by_{2} \\
    \vdots \\
    \by_{n-1} \\
    -\by_1-\by_2-\dots -\by_{n-1} \\
\end{bmatrix}, \;\;\;\;\;\;\;\by_i \in\reals^k.\\
\end{align*}
Using Lemma \ref{lemma:randomspectrumlaplace}, $\bz_{\parallel}$ is in the nullspace of 
$\left(\bD_{G_n}-\bA_{G_n}\right)\otimes \id_{k_n}$. Therefore,
\begin{align*}
\left\langle \bz, \tilde \bY_n \bz\right\rangle & = \left\langle\bz_{\parallel},\left(\bA_{G_n}\otimes\tilde \bA_{H_n}\right)\bz_{\parallel}\right\rangle  \\
& \;\;\;\;+2\left\langle\bz_{\perp},\left(\bA_{G_n}\otimes\tilde \bA_{H_n}\right)\bz_{\parallel}\right\rangle \\
& \;\;\;\;+ \left\langle\bz_{\perp},\left(2e(G_n)\tilde a_n(\bD_{G_n}-\bA_{G_n})\otimes \id_{k_n} + \bA_{G_n}\otimes \tilde \bA_{H_n}\right)\bz_{\perp}\right\rangle.
\end{align*}
Note that
\begin{align*}
\left\langle\bz_{\parallel},\left(\bA_{G_n}\otimes\tilde \bA_{H_n}\right)\bz_{\parallel}\right\rangle &= 2e(G)\left\langle \bx,\tilde \bA_{H_n} \bx\right\rangle,\\
\left\langle\bz_{\perp},\left(\bA_{G_n}\otimes\tilde \bA_{H_n}\right)\bz_{\parallel}\right\rangle &= \sum_{i=1}^{n-1}(\deg(i)-\deg(n))\left\langle \by_i,\tilde \bA_{H_n} \bx\right\rangle.
\end{align*}
Also using Lemmas \ref{lemma:randomspectrum}, \ref{lemma:randomspectrumlaplace} we have
\begin{align*}
\left\langle\bz_{\perp},\tilde \bY_n\bz_{\perp} \right\rangle\ge \left(2nq_0\tilde a_ne(G_n)-2(1+o(1))\sqrt{nq_0(1-q_0)}\left(2\tilde a_ne(G_n)+\lambda_{1}(\tilde \bA_{H_n})\right)\right)\|\bz_{\perp}\|^2.
\end{align*}
Note that if $\lambda_{k_n}(\bA_{H_n}) \ge (2e(H_n)-k_n^2)/k_n$,
\begin{align*}
-\lambda_{k_n}(\bA_{H_n})-1+ \frac{k_n\lambda_{k_n}(\bA_{H_n})+k_n^2-2e(H_n)}{nk_n^2} \ge -1.
\end{align*}
Therefore,
\begin{align*}
\lambda_{1}(\tilde \bA_{H_n}) &= \lambda_{1}(\bA_{H_n}) + u_n + 2a_ne(G_n)\\
&\le \lambda_{1}(\bA_{H_n})-\lambda_{k_n}(\bA_{H_n}) + \frac{k_n\lambda_{k_n}(\bA_{H_n})+k_n^2-2e(H_n)}{nk_n^2} \\
& \;\;\;\;+ \frac{2e(H_n)+k_nu_n+nk_n^2-k_n^2}{k_n(nk_n-1)} \\
&\le \lambda_{1}(\bA_{H_n})-\lambda_{k_n}(\bA_{H_n}) + 1.
\end{align*}
Otherwise, note that 
\begin{align*}
\lambda_{1}(\bA_{H_n}) \ge \frac{2e(H_n)}{k_n}.
\end{align*}
Therefore, if $\lambda_{k_n}(\bA_{H_n}) < (2e(H_n)-k_n^2)/k_n$ then
\begin{align*}
\lambda_{1}(\bA_{H_n}) - \lambda_{k_n}(\bA_{H_n}) \ge k_n \ge \lambda_{1}(\tilde \bA_{H_n}).
\end{align*}
Thus,
\begin{align*}
\lambda_{1}(\tilde \bA_{H_n}) \le \lambda_{1}(\bA_{H_n})-\lambda_{k_n}(\bA_{H_n}) + 1.
\end{align*}
In both cases. Using the fact that, 
$\limsup_{n\to \infty} {(2(\lambda_{1}(\bA_{H_n})-\lambda_{k_n}(\bA_{H_n}))\sqrt{1-q_0})/{\sqrt{nq_0}}}=1-C$,
\begin{align*}
\left\langle\bz_{\perp},\tilde \bY_n\bz_{\perp}\right\rangle \ge 2C\tilde a_ne(G_n)nq_0\|\bz_{\perp}\|^2.
\end{align*}
Thus, in order to show positive semidefiniteness of $\tilde \bY_n$, it suffices to show that
\begin{align*}
\scriptsize
\bY^\prime_n = \begin{bmatrix}
    2e(G_n)\tilde \bA_{H_n} & (\deg(1)-\deg(n))\tilde \bA_{H_n} & (\deg(2)-\deg(n))\tilde \bA_{H_n} & \cdots & (\deg(n-1)-\deg(n))\tilde \bA_{H_n}\\
    (\deg(1)-\deg(n))\tilde \bA_{H_n} & 2C\tilde a_ne(G_n)nq_0\id_{k_n} & 0 & \cdots & 0\\
    (\deg(2)-\deg(n))\tilde \bA_{H_n} & 0 & 2C\tilde a_ne(G_n)nq_0\id_{k_n} & \cdots & 0\\
    \vdots & \vdots & \vdots & \ddots & \vdots\\
    (\deg(n-1)-\deg(n))\tilde \bA_{H_n} & 0 &0& \cdots & 2C\tilde a_ne(G_n)nq_0\id_{k_n}\\
\end{bmatrix} \succeq 0.
\end{align*}
Note that using Bernstein's inequality,
\begin{align*}
\P \left\{\left|\max_{i\in V(G_n)}\deg(i) - nq_0\right| \geq 2\sqrt{nq_0\log n}\right\}
&\leq 2n\exp\left\{\frac{2nq_0\log n}{nq_0(1-q_0)+(2/3)\sqrt{nq_0\log n}} \right\}\\
&\leq 2\exp\left\{\log n - 2\log n\right\} = \frac{2}{n},
\end{align*}
for large enough $n$. Thus,
\begin{align*}
\P \left\{\left|\max_{i\in V(G_n)}\deg(i) - nq_0\right| \geq 2\sqrt{nq_0\log n}\right\} \to 0,
\end{align*}
as $n \to \infty$. 
Hence, $\deg(i)-\deg(n)\le4\sqrt{nq_0\log n}$, for all $n$, with high probability, as $n \to \infty$. 
Therefore, using Schur's theorem, since $C>0$, we need to show that
\begin{align}
2e(G_n)\tilde \bA_{H_n} - 16\left(2C\tilde a_ne(G_n)nq_0\right)^{-1}n^2q_0\log n\tilde \bA_{H_n}^2\nonumber\\
= C^\prime\tilde\bA_{H_n}\left(\frac{C\tilde a_n (e(G_n))^2}{4n\log n}\id_{k_n}- \tilde \bA_{H_n}\right)\succeq 0.
\end{align}
Where $C^\prime >0$. This holds, since $\tilde \bA_{H_n}\succeq 0$. Further,
$C\tilde a_n (e(G_n))^2/(4n\log n)$
is $\Theta(n/\log n)$ and
\begin{align*}
\lim_{n\to\infty} \frac{(\lambda_{1}(\bA_{H_n})-\lambda_{k_n}(\bA_{H_n}))\log n}{n} = 0.
\end{align*}
Hence,
$\bY_n$ is feasible for problem \eqref{equation:SDP}, with high probability as $n\to \infty$. Now, note that
\begin{align*}
\Tr\left((\bA_{G_n}\otimes \bA_{H_n})\bY_n\right) \ge 4e(H_n)e(G_n)b_n = 2e(H_n).
\end{align*}
Thus, with high probability as $n\to \infty$ under null, the optimal value of 
problem \eqref{equation:SDP}, ${\sf SDP}(G_n;H_n)$, is bigger than or equal $2e(H_n)$.
Note that the optimal value of \eqref{equation:SDP} under 
the alternative when there is no noise
is $2e(H_n)$. Therefore, under the conditions of the Theorem \ref{thm:convexgeneral},
for the test based on ${\sf SDP}(G_n;H_n)$,
\begin{align*}
\P_{0,n}\{T(G_n)=1\} \to 1,
\end{align*}
as $n\to \infty$ and the proof is complete.
\end{proof}
\section{Proofs: Multiple planted subgraphs}

\subsection{Proof of Theorem \ref{thm:lowerbdmultiple}}

In order to state the proofs in this section we will use the following notation. Let 
$\varphi_1,\varphi_2,\dots,\varphi_m$ be labelings of $V(H_n)$ in $[n]$. We set
\begin{align*}
M(\varphi_1,\varphi_2,\dots,\varphi_m) \equiv \left|\bigcup_{l=1}^m \varphi_l(E(H_n)) \right|.
\end{align*}
Also recall
\begin{align*}
N(H_n;G_n) \equiv \Big|\big\{\,\varphi\in\cL(H_n;n):\;\; \varphi(E(H_n))\subseteq E(G_n)\,
\big\}\Big| \, .
\end{align*}
We also denote $v(H_n)$ by $k_n$.
First we prove the following Lemma which is useful in the proof of Theorem \ref{thm:lowerbdmultiple}.

\begin{lemma}
\label{lemma:Nconcentrationmultiple}
Let $\{H_n\}_{n\ge1}, q_0, \P_{0,n}, \P_{1,n}$ be as in Theorem \ref{thm:lowerbdmultiple}.
Under the assumptions of Theorem \ref{thm:lowerbdmultiple}, for all $\eps>0$
\begin{align}\label{eq:Nconcentrationmultiple}
\lim_{n\to \infty} \P_{0,n}\left\{N(H_n;G_n)^{m_n} \le (1-\eps)(\E_{0,n}N(H_n;G_n))^{m_n}\right\} = 0.
\end{align}
\end{lemma}
\begin{proof}
As in \eqref{eq:chebyshev} in the proof of Lemma \ref{lemma:Nconcentrationp1}, using Chebyshev's inequality
\begin{align}
\label{eq:ntomnoverentomn}
&\P\left\{ \frac{N(H_n;G_n)^{m_n}}{(\E_{0,n}N(H_n;G_n))^{m_n}}\leq 1-\eps\right\} =
\P\left\{\frac{N(H_n;G_n)}{\E_{0,n}N(H_n;G_n)}\leq (1-\eps)^{1/m_n}\right\}\nonumber\\
&\leq\P\left\{\frac{N(H_n;G_n)}{\E_{0,n}N(H_n;G_n)}\leq \Big(1-\frac{\eps}{m_n}\Big)\right\}\leq
\frac{m_n^2\bar\Delta(n,H_n)}{\eps^2(\E_0 N(H_n;G_n))^2}.
\end{align}
Where $\bar\Delta(n,H_n)$ is defined in \eqref{eq:deltadefinition}. 
Note that, using \eqref{eq:deltaovern} we have
\begin{align*}
\frac{m_n^2\bar\Delta(n,H_n)}{(\E_0 N(H_n;G_n))^2} \leq
C\exp\left\{-\tilde u\left(f(\tilde u)-\frac{\log k_n}{\tilde u}-\frac{2\log m_n}{\tilde u}\right)\right\}
\end{align*}
where $f(u)$ is defined in \eqref{eq:fofu} and \eqref{eq:fofk} and $\tilde u = \arg\min_{2\leq u\leq k_n}\{uf(u)\}$.
Using \eqref{eq:fofuminuslogkoveru},
\begin{align*}
f(\tilde u)-\frac{\log k_n}{\tilde u}-\frac{2\log m_n}{\tilde u} \geq \log n - \frac{e_{H_n}(\tilde u)}{\tilde u}\log(1/q_0)-(5/2)\log k_n-\frac{2\log m_n}{\tilde u} + C
\end{align*}
for some constant $C$. Note that $0\leq e_{H_n}(\tilde u)/\tilde u\leq \min(d(H_n),\tilde u)$. Therefore,
\begin{align}
\label{eq:ftildeuminus}
f(\tilde u)-\frac{\log k_n}{\tilde u}-\frac{2\log m_n}{\tilde u} \geq \log n - \min\left(d(H_n),\tilde u\right)\log(1/q_0) -\frac{2\log m_n}{\tilde u}  -(5/2)\log k_n + C.
\end{align}
The right hand side of \eqref{eq:ftildeuminus}, as a function of
$\tilde u$, is minimized at $\tilde u = 2$ or $\tilde u = d(H_n)$
(we can assume, without loss of generality, $d(H_n)\ge 2$). Therefore,
it is sufficient to show that the right hand side of \eqref{eq:ftildeuminus} goes to infinity for $\tilde u \in\{ 2,d(H_n)\}$.
For $\tilde u = 2$, the right hand side of \eqref{eq:ftildeuminus} is equal to
\begin{align}
\label{eq:Xforuequal2}
X = \log n\left(1-\frac{2}{\log n}\log(1/q_0) - \frac{\log m_n}{\log n} -(5/2)\frac{\log k_n}{\log n} + o(1)\right) \to \infty
\end{align}
since by the assumption of Theorem \ref{thm:lowerbdmultiple}
\begin{align*}
\lim\sup_{n\to \infty}\frac{(5/2)\log k_n+\log m_n}{\log n}<1.
\end{align*}
For $\tilde u = d(H_n)$, the right hand side of \eqref{eq:ftildeuminus} is equal to
\begin{align*}
X = \log n\left(1-\frac{d(H_n)}{\log n}\log(1/q_0) - \frac{2\log m_n}{d(H_n)\log n} -(5/2)\frac{\log k_n}{\log n}+o(1)\right).
\end{align*}
Fix a constant  $M>0$ large enough (to be adjusted below). For $d(H_n) \leq M$ we have
\begin{align}
\label{eq:dlessthanC}
X &\geq \log n\left(1-\frac{M\log(1/q_0)}{\log n} - \frac{2\log m_n}{d(H_n)\log n} - (5/2)\frac{\log k_n}{\log n}+o(1)\right)\nonumber\\
&\geq \log n\left(1-\frac{2}{M} - \frac{\log m_n}{\log n} - (5/2)\frac{\log k_n}{\log n}+o(1)\right)
\end{align}
where the last inequality holds for all  $n$ large enough. For $d(H_n) \geq M$, using $m_n\leq n$, we have 
\begin{align}
\label{eq:dbiggerthanC}
X &\geq \log n\left(1-\frac{d(H_n)\log(1/q_0)}{\log n} - \frac{2\log m_n}{M\log n} - (5/2)\frac{\log k_n}{\log n}+o(1)\right)\nonumber\\
&\geq \log n\left(1- \frac{d(H_n)\log (1/q_0)}{\log n} - \frac{2}{M} - (5/2)\frac{\log k_n}{\log n}+o(1)\right).
\end{align}
Combining \eqref{eq:dbiggerthanC}, \eqref{eq:dlessthanC}, we have
\begin{align}
\label{eq:Xbiggerthandeltaminus2overC}
X &\geq \log n\left(\min\left(1-\frac{\log m_n}{\log n} - (5/2)\frac{\log k_n}{\log n}, 1- \frac{d(H_n)\log (1/q_0)}{\log n} - (5/2)\frac{\log k_n}{\log n}\right) - \frac{2}{M} - o(1)\right)\nonumber\\
&\geq \left(\delta - \frac{2}{M}\right)\log n.
\end{align}
Where
\begin{align*}
\delta = \min\left(\liminf_{n\to \infty}\left\{1-\frac{\log m_n}{\log n} - (5/2)\frac{\log k_n}{\log n}\right\},\liminf_{n\to \infty}\left\{1- \frac{d(H_n)\log (1/q_0)}{\log n} - (5/2)\frac{\log k_n}{\log n}\right\}\right)
\end{align*}
and by assumptions of Theorem \ref{thm:lowerbdmultiple},
$\delta>0$. Hence, by taking $M = 4/\delta$, 
using \eqref{eq:Xbiggerthandeltaminus2overC}, we deduce that the right hand side of \eqref{eq:ftildeuminus}
goes to infinity for $\tilde u = d(H_n)$. Combining with \eqref{eq:Xforuequal2}, we deduce that 
under the assumptions of Theorem \ref{thm:lowerbdmultiple}
\begin{align*}
f(\tilde u)-\frac{\log k_n}{\tilde u}-\frac{2\log m_n}{\tilde u} \to \infty
\end{align*}
and
\begin{align}
\label{eq:msquaredeltaovernsquare}
\frac{m_n^2\bar\Delta(n,H_n)}{(\E_0 N(H_n;G_n))^2}\to 0.
\end{align}
as $n \to \infty$.
Hence, using \eqref{eq:ntomnoverentomn}, \eqref{eq:msquaredeltaovernsquare}
we get that under the assumptions of Theorem \ref{thm:lowerbdmultiple}, 
\begin{align*}
\lim_{n\to \infty} \P_{0,n}\left\{(N(H_n;G_n))^{m_n} \le (1-\eps)(\E_0N(H_n;G_n))^{m_n}\right\} = 0
\end{align*}
This completes the proof of lemma.
\end{proof}
Now we can state the proof of Theorem \ref{thm:lowerbdmultiple}.
\begin{proof}[Proof of Theorem \ref{thm:lowerbdmultiple}]
For the laws $\P_{0,n}$, $\P_{1,n}$ defined as in Theorem \ref{thm:lowerbdmultiple} we have
\begin{align}
\begin{split}
\label{equation:RNderivative}
\frac{\d \P_{1,n}}{\d \P_{0,n}}(G_n) &= \frac{1}{(n)_{k_n}^{m_n}}\sum_{\varphi_1,\dots,\varphi_{m_n}\in \mathcal{L}(H_n,n)}
\left(\frac{1}{q_0}\right)^{M\left(\varphi_1,\dots,\varphi_{m_n}\right)}
\mathbb I\left\{\left(\bigcup_l\varphi_l(E(H_n))\right)\subseteq E(G_n)\right\}\\
&= \frac{1}{(n)_{k_n}^{m_n}q_0^{m_ne(H_n)}}\sum_{\varphi_1,\dots,\varphi_{m_n}\in \mathcal{L}(H_n,n)}
\mathbb I\left\{\left(\bigcup_l\varphi_l(E(H_n))\right)\subseteq E(G_n)\right\}\\
&- \frac{1}{(n)_{k_n}^{m_n}q_0^{m_ne(H_n)}}\sum_{\varphi_1,\dots,\varphi_{m_n}\in \mathcal{L}(H_n,n)}
\left(1-q_0^{m_ne(H_n)-M(\varphi_1,\varphi_2,\dots,\varphi_{m_n})}\right)\mathbb I\left\{\left(\bigcup_l\varphi_l(E(H_n))\right)\subseteq E(G_n)\right\}\\
&= \frac{N(H_n;G_n)^{m_n}}{\left(\E_{0,n}N(H_n;G_n)\right)^{m_n}}- X_n.
\end{split}
\end{align}
where  
\begin{align*}
X_n = \frac{1}{(n)_{k_n}^{m_n}q_0^{m_ne(H_n)}}\sum_{\varphi_1,\dots,\varphi_{m_n}\in \mathcal{L}(H_n,n)}
\left(1-q_0^{m_ne(H_n)-M(\varphi_1,\varphi_2,\dots,\varphi_{m_n})}\right)\mathbb I\left\{\left(\bigcup_l\varphi_l(E(H_n))\right)\subseteq E(G_n)\right\}.
\end{align*} 
Note that under the assumptions of Theorem \ref{thm:lowerbdmultiple},
using Lemma \ref{lemma:Nconcentrationmultiple}
\begin{align*}
\lim_{n\to \infty} \P_{0,n}\left\{(N(H_n;G_n))^{m_n} \le (1-\eps)(\E_0N(H_n;G_n))^{m_n}\right\} = 0.
\end{align*}
Hence, using the same argument used in proof of Theorem \ref{thm:lowerbdp1},
\begin{align}
\label{eq:Nmnoverexpnmnto1}
\frac{N(H_n;G_n)^{m_n}}{\E_{0,n}(N(H_n;G_n)^{m_n})}\toprob 1.
\end{align}
Now, we prove that under the assumptions of Theorem \ref{thm:lowerbdmultiple},
$X_n \toprob$ 0. Note that we have
\begin{align}
\label{eq:xnlessthan}
X_n &\leq \frac{1}{(\E_{0,n}N(H_n;G_n))^{m_n}}\sum_{\varphi_1,\dots,\varphi_{m_n}\in \mathcal L(H_n,n)}\mathbb I\left\{M(\varphi_1,\dots,\varphi_{m_n}) < m_ne(H_n)\right\}\\
&=
\frac{(N(H_n;G_n))^{m_n}}{(\E_{0,n}N(H_n;G_n))^{m_n}}\frac{1}{(N(H_n;G_n))^{m_n}}\sum_{\varphi_1,\dots,\varphi_{m_n}\in
  \mathcal L(H_n,n)}\mathbb I\left\{M(\varphi_1,\dots,\varphi_{m_n}) <
  m_ne(H_n)\right\}\\
& \equiv \frac{(N(H_n;G_n))^{m_n}}{(\E_{0,n}N(H_n;G_n))^{m_n}} \,
Q_{m_n}(H_n;G_n)\, .\label{eq:Qn}
\end{align}
Note that $Q_{m_n}(G_n;H_n)$ can be interpreted as the probability
that, drawing the embeddings $\varphi_{l}$, $l\in\{1,\dots,m_m\}$ independently and
uniformly at random in $\cL(H_n,h)$, at least two of them share an
edge.
By union bound, we have 
\begin{align}
Q_{m_n}(H_n;G_n)&\le m_n^2 \, Q_2(H_n;G_n)\, ,\\
Q_{2}(H_n;G_n) &= \frac{N_2(H_n;G_n)}{N(H_n;G_n)^2}\, ,\label{eq:Qn_Q2}
\end{align}
and 
\begin{align*}
N_2(H_n;G_n)= \left|\left\{(\varphi_1,\varphi_2);\varphi_1(E(H_n))\cap\varphi_2(E(H_n))\neq \emptyset, \varphi_i(E(H_n))\subseteq E(G_n) \;\; \text{for}\;\; i=1,2 \right\}\right|.
\end{align*}
Note that 
\begin{align*}
\E_{0,n}N_2(H_n;G_n) \leq \sum_{l=2}^{k_n}n^{2k_n-l}q_0^{2e(H_n)-e_{H_n}(l)}\leq k_nn^{2k_n-\tilde l}q_0^{2e(H_n)-e_{H_n}(
\tilde l)}
\end{align*}
where $\tilde l = \arg\max_{2\leq l\leq k_n}n^{2k_n-l}q_0^{2e(H_n)-e_{H_n}(l)}$. Hence,
\begin{align*}
\log\frac{m_n^2\E_{0,n}N_2(H_n;G_n)}{(\E_{0,n} N(H_n;G))^2} &\leq -\tilde l\left(\log n - \frac{\log k_n}{\tilde l} - \frac{2\log m_n}{\tilde l} - \frac{e_{H_n}(\tilde l)}{\tilde l}\log(1/q_0)\right)\\
&\leq -\tilde l\left(\log n - \min\left(\tilde l,d(H_n)\right)\log(1/q_0) - \frac{\log k_n}{\tilde l} - \frac{2\log m_n}{\tilde l}\right).
\end{align*}
Note that in the proof of Lemma \ref{lemma:Nconcentrationmultiple}, we proved that 
under the assumptions of Theorem \ref{thm:lowerbdmultiple}, for all $2\leq l\leq k_n$
\begin{align*}
\min_{2\le l\le k_n}\left[\log n - \min\left(d(H_n), l\right)\log(1/q_0) -\frac{2\log
    m_n}{l}  -(5/2)\log k_n\right] \to \infty
\end{align*}
as $n \to \infty$. Hence, under the assumptions of Theorem \ref{thm:lowerbdmultiple}, 
\begin{align*}
-\tilde l\left(\log n - \min\left(\tilde l,d(H_n)\right)\log(1/q_0) - \frac{\log k_n}{\tilde l} - \frac{2\log m_n}{\tilde l}\right) \to -\infty
\end{align*}
and 
\begin{align}
\label{eq:mnsquaren2overnsquare}
\frac{m_n^2\E_{0,n}N_2(H_n;G_n)}{(\E_{0,n}N(H_n;G_n))^2}\to 0.
\end{align}
Note that
\begin{align*}
m_n^2Q_2(H_n;G_n) = \frac{m_n^2N_2(H_n;G_n)}{(\E_{0,n}N(H_n;G_n))^2}\frac{(\E_{0,n}N(H_n;G_n))^2}{N(H_n;G_n)^2}.
\end{align*}
Using \eqref{eq:mnsquaren2overnsquare} and Markov's inequality
\begin{align*}
\frac{m_n^2N_2(H_n;G_n)}{(\E_{0,n}N(H_n;G_n))^2} \toprob 0.
\end{align*}
Hence, by \eqref{eq:Qn}, \eqref{eq:Qn_Q2}, and \eqref{eq:xnlessthan} 
\begin{align*}
X_n \toprob 0.
\end{align*}
Thus, using \eqref{equation:RNderivative} and \eqref{eq:Nmnoverexpnmnto1}, under $\P_{0,n}$
\begin{align*}
\frac{\d\P_{1,n}}{\d\P_{0,n}}(G_n)\toprob 1
\end{align*}
and using Lemma \ref{lemma:likelihoodtest} the proof is complete.
\end{proof}

\subsection{Proof of Theorem \ref{thm:upperbdmultiplebdd}}
\begin{proof}[Proof of Theorem \ref{thm:upperbdmultiplebdd}]
Let $N_{0,n}, N_{1,n}$ be the random variables denoting the number of edges of $G_n$ 
when it is generated according to null and alternative models, respectively. We would like to show that
there exists $t_n^*$ such that 
$N_{0,n} < t_n^*$ and $N_{1,n} > t_n^*$ with high probability as $n \to \infty$.
Let
\begin{align*}
t_n^* = \frac{n(n-1)q_0}{2}+\delta_nm_ne(H_n)
\end{align*}
for some $\delta_n>0$. Note that $N_{0,n}$ is a binomial 
random variable where $\E N_{0,n} = n(n-1)q_0/2$ and $\Var(N_{0,n}) = n(n-1)q_0(1-q_0)/2$. Hence, 
using Chebyshev's inequality
\begin{align*}
N_{0,n} \leq \frac{n(n-1)q_0}{2} + C_n\sqrt{\frac{n(n-1)q_0(1-q_0)}{2}}
\end{align*}
with probability $1- 1/C_n^2$.
Since $\lim\inf_{n\to \infty} m_ne(H_n)/n=\infty$,
taking $C_n\to \infty$ such that $C_n/\delta_n < m_ne(H_n)/n$, we have
$N_{0,n} < t^*_n$ with high probability as $n\to \infty$.
Note that $N_{1,n}$ is monotonically increasing in $m_n$. Hence, in
order to show  that $N_{1,n} > t_n^*$ as $n\to \infty$ we can assume,
without loss of generality,  
$\lim\sup _{n\to\infty}m_ne(H_n)/n^2 = 0$. We have 
\begin{align}
\label{eq:edgecountalter}
N_{1,n} = X_n + m_ne(H_n)-Z_{0,n}-\sum_{i=1}^{m_n-1}iZ_{i,n}.
\end{align}
Here $X_n$ is denotes the number of 
edges in the graph before adding the copies of $H_n$. Hence, $\E X_n = n(n-1)q_0/2$ and $\Var(X_n) = n(n-1)q_0(1-q_0)/2$.
Further, $m_ne(H_n)$ is the total number of
edges of the planted subgraphs; $Z_{0,n}$ is the number of edges in the planted subgraphs
that are present before adding the subgraphs and
$Z_{i,n}$ is the number of edges that are present in exactly $i+1$ different embeddings. Using Chebyshev's inequality
\begin{align}
\label{eq:XandZchebyshev}
&X_n\geq \frac{n(n-1)q_0}{2} - \tilde
C_n\sqrt{\frac{n(n-1)q_0(1-q_0)}{2}}\;\;\;\;\text{with probability at
  least $1- \frac{1}{\tilde C_n^2}$},\\
&Z_{0,n}\leq m_ne(H_n)q_0 + \tilde
C_n\sqrt{\frac{m_ne(H_n)q_0(1-q_0)}{2}}\;\;\;\;\text{with probability
  at least $1- \frac{1}{\tilde C_n^2}$}.
\end{align}
(The last inequality follows because $Z_{0,n}$ is dominated by a
binomial random variable with parameters $m_ne(H_n)$, $q_0$.)

In addition, 
\begin{align*}
\sum_{i=1}^{m_n-1} i\E Z_{i,n} &\leq \sum_{i=1}^{m_n-1}\frac{i{m_n\choose{i+1}}(e(H_n))^{i+1}} {{n\choose 2}^i} 
\leq \sum_{i=1}^{m_n-1}\frac{i(m_n e(H_n))^{i+1}}{n^{2i}} \\
&\leq \frac{(m_ne(H_n))^2}{(n^2-m_ne(H_n))(1-m_ne(H_n)/n^2)}\leq \frac{2(m_ne(H_n))^2}{n^2-m_ne(H_n)}
\end{align*}
for large enough $n$. Therefore, using Markov's inequality,
\begin{align}
\label{eq:Zimarkov}
\sum_{i=1}^{m_n-1}iZ_{i,n} \leq \eps m_ne(H_n) \;\;\;\;\text{with probability $1-\frac{2m_ne(H_n)}{\eps(n^2-m_ne(H_n))}$}.
\end{align}
Now using \eqref{eq:edgecountalter}-\eqref{eq:Zimarkov}, we can write
\begin{align*}
N_{1,n} \geq \frac{n(n-1)q_0}{2} + (1-q_0-\eps)m_ne(H_n)-\tilde C_n\sqrt{\frac{n(n-1)q_0(1-q_0)}{2}} - \tilde C_n\sqrt{\frac{m_ne(H_n)q_0(1-q_0)}{2}}
\end{align*}
with probability $1- 2/\tilde C_n^2-3m_ne(H_n)/(\eps n^2)$.
Since $\lim\inf_{n\to \infty} m_ne(H_n)/n=\infty$,
taking $\eps < 1-q_0-\delta_n$ and $\tilde C_n \to \infty$ such that $\tilde C_n/(1-q_0-\eps-\delta_n) < m_ne(H_n)/(n+\sqrt{m_ne(H_n)})$,
we have
\begin{align*}
N_{1,n} > t^*_n
\end{align*}
with high probability as $n\to\infty$. Therefore, the two models are strongly distinguishable 
under the assumptions of Theorem \ref{thm:upperbdmultiplebdd} and this completes the proof.
\end{proof}

\end{document}